\documentclass[11pt,twoside,reqno]{amsart}

\usepackage{microtype}
\usepackage[OT1]{fontenc}
\usepackage{type1cm}
\usepackage{amssymb}

\usepackage[left=2.3cm,top=3cm,right=2.3cm]{geometry}
\usepackage{color}
\usepackage{amsaddr}

\newtheorem{theorem}{Theorem}[section]
\theoremstyle{plain}

\newtheorem{cor}[theorem]{Corollary}

\newtheorem{example}[theorem]{Example}

\newtheorem{lemma}[theorem]{Lemma}

\newtheorem{prop}[theorem]{Proposition}

\numberwithin{equation}{section}

\newcommand{\ly}[1]{\chi^{#1}_{\nu}}
\newcommand{\R}{\mathbb{R}}
\newcommand{\N}{\mathbb{N}}

\newcommand{\proj}{\mathrm{proj}}

\newcommand{\av}{\underline{a}}
\newcommand{\xv}{\underline{x}}
\newcommand{\yv}{\underline{y}}
\newcommand{\ii}{\mathbf{i}}
\newcommand{\iiv}{\overline{\imath}}
\newcommand{\jjv}{\overline{\jmath}}

\newcommand{\jj}{\mathbf{j}}
\newcommand{\tv}{\underline{t}}

\newcommand{\pv}{\underline{p}}

\newcommand{\Alpha}{\mathcal{A}}
\newcommand{\cm}[2]{\lambda^{#2}_{#1}}
\newcommand{\cmd}[2]{\mu^{#2}_{#1}}
\newcommand{\lye}[2]{\chi^{#2}_{#1}}
\newcommand{\y}{\underline{\mathbf{y}}}
\newcommand{\z}{\underline{\mathbf{z}}}

\newcommand{\dd}{\,\mathrm{d}}

\DeclareMathOperator*{\esssup}{ess\,sup}
\DeclareMathOperator*{\essinf}{ess\,inf}

\DeclareMathOperator{\dimh}{dim_H}
\DeclareMathOperator{\udimh}{\overline{dim}_H}
\DeclareMathOperator{\ldimh}{\underline{dim}_H}

\DeclareMathOperator{\dimp}{dim_p}
\DeclareMathOperator{\udimp}{\overline{dim}_p}
\DeclareMathOperator{\ldimp}{\underline{dim}_p}

\DeclareMathOperator{\dimloc}{dim_{loc}}
\DeclareMathOperator{\udimloc}{\overline{dim}_{loc}}
\DeclareMathOperator{\ldimloc}{\underline{dim}_{loc}}

\begin{document}
\title[Ledrappier-Young formula and exact dimensionality of self-affine measures]{Ledrappier-Young formula and exact dimensionality of self-affine measures}

\author{Bal\'azs B\'ar\'any}
\address[Bal\'azs B\'ar\'any]{ Mathematics Institute, University of Warwick, Coventry CV4 7AL, United Kingdom \& \\
	Budapest University of Technology and Economics, MTA-BME Stochastics Research Group, P.O.\ Box 91, 1521 Budapest, Hungary} \email{balubsheep@gmail.com}

\author{Antti K\"aenm\"aki}
\address[Antti K\"aenm\"aki]{Institute of Mathematics,
Polish Academy of Sciences, S\'niadeckich 8, 00-656 Warsaw, Poland \&
Department of Mathematics and Statistics, P.O.\ Box 35 (MaD), FI-40014
University of Jyv\"askyl\"a, Finland} \email{antti.kaenmaki@jyu.fi}

\thanks{This work was partially supported by the grant 346300 for IMPAN from the Simons Foundation and by the Polish MNiSW 2015-2019 matching fund. B\'ar\'any also acknowledges the support from the grants EP/J013560/1 and OTKA K104745 and the J\'anos Bolyai Research Scholarship of the Hungarian Academy of Sciences.}
\date{\today}

\begin{abstract}
  In this paper, we solve the long standing open problem on exact dimensionality of self-affine measures on the plane. We show that every self-affine measure on the plane is exact dimensional regardless of the choice of the defining iterated function system. In higher dimensions, under certain assumptions, we prove that self-affine and quasi self-affine measures are exact dimensional. In both cases, the measures satisfy the Ledrappier-Young formula.
\end{abstract}

\maketitle

\thispagestyle{empty}

\section{Introduction}

Let $\Alpha=\left\{A_1,\dots, A_N\right\}$ be a finite set of contracting non-singular $d\times d$ matrices, and let $\Phi=\left\{f_i(\xv)=A_i\xv+\tv_i\right\}_{i=1}^N$ be an iterated function system (IFS) of affine mappings, where $\underline{t}_i\in\R^d$ for all $i \in \{ 1,\dots,N \}$. It is a well known fact that there exists a unique non-empty compact subset $\Lambda$ of $\R^d$ such that $$\Lambda=\bigcup_{i=1}^Nf_i(\Lambda).$$ We call the set $\Lambda$ a \textit{self-affine set} associated to $\Phi$.

Let $\alpha_i(A)$ be the $i$th singular value of a $d\times d$ non-singular matrix $A$. Namely, $\alpha_i(A)$ is the positive square root of the $i$th eigenvalue of $AA^*$, where $A^*$ is the transpose of $A$. Thus, $0<\alpha_d(A)\leq\cdots\leq\alpha_1(A)<1$. The geometric interpretation of the singular values is that the linear map $\xv\mapsto A\xv$ maps the $d$-dimensional unit ball to an ellipse with semiaxes of length $\alpha_d(A)\leq\cdots\leq\alpha_1(A)$.
For a subspace $V\subseteq\R^d$, we define the restricted operator norm of a matrix $A$ to be
$$
\|A|V\|=\sup_{v\in V}\frac{\|Av\|}{\|v\|},
$$
where $\|v\|$ denotes the Euclidean norm of a vector $v$. Let $\mathfrak{m}(A|V)=\|A^{-1}|V\|^{-1}$ and note that $\alpha_1(A)=\|A|\R^d\|$ and $\alpha_d(A)=\mathfrak{m}(A|\R^d)$.

We denote the Hausdorff dimension of a set $E\subseteq\R^d$ by $\dimh E$ and the packing dimension by $\dimp E$. For the definitions and basic properties of these quantities, we refer to Falconer~\cite{Fb1}.

The dimension theory of self-affine sets is far away from being well understood. Bedford~\cite{Be} and McMullen~\cite{Mc} studied the Hausdorff and packing dimensions of a carpet-like planar class of self-affine sets. This result was generalised by Kenyon and Peres~\cite{KP} for higher dimensions. Later, Gatzouras and Lalley~\cite{GL} and Bara\'nski~\cite{B} studied a more general class of carpet-like self-affine sets in the plane. Fraser~\cite{Fra} has calculated the packing dimension for general box-like planar self-affine sets.

Falconer~\cite{F} introduced the singular value function for non-singular matrices and defined the corresponding subadditive pressure. He showed that the zero of the pressure, the affinity dimension, is an upper bound for the packing dimension of the self-affine set. He also proved that if the contraction ratios of the mappings are strictly less than $1/3$ then the Hausdorff and packing dimensions coincide and equal to the affinity dimension for Lebesgue almost every translation vector $(\tv_1,\dots,\tv_N)$. Later, Solomyak~\cite{S} extended the bound to $1/2$. It follows from the example of Przytycki and Urba\'nski~\cite{PU} that the bound $1/2$ is sharp. Very recently, working on the plane, B\'ar\'any and Rams~\cite{BR} proved that for almost all positive matrices under the strong separation condition, the Hausdorff dimension is equal to the affinity dimension provided that the $1$-bunched condition holds or the affinity dimension is greater than $5/3$. In the overlapping carpet case, Shmerkin~\cite{Sh} has used the transversality method for self-similar sets to calculate the dimension of a class of box-like self-affine sets. Furthermore, Fraser and Shmerkin~\cite{FS} have shown that the dimension of a typical overlapping Bedford-McMullen carpet-like self-affine set is equal to the dimension of the corresponding non-overlapping Bedford-McMullen carpets.

The first dimension result valid for open set of translation vectors was given by Hueter and Lalley~\cite{HL}. They showed that under some conditions on the matrices, if the strong separation condition holds, then the Hausdorff dimension of the self-affine set is equal to the affinity dimension, which in this case is less than $1$. K\"aenm\"aki and Shmerkin~\cite{KS} proved a similar statement for the packing dimension of overlapping self-affine sets of Kakeya-type, in which case the dimension is strictly larger than $1$. Falconer~\cite{F2} gave a condition on the projection of the self-affine set, under which the packing dimension is equal to the affinity dimension. Falconer and Kempton~\cite{FK1} generalised this result (and the condition) on the plane for the Hausdorff dimension.

Let us next consider the dimension theory of self-affine measures. Let $\mu$ be an arbitrary Radon measure on $\R^d$ and denote by $B(\xv,r)$ the $d$-dimensional closed ball centered at $\xv\in\R^d$ with radius $r$. Then we call
\begin{equation*}
\ldimloc(\mu,\xv)=\liminf_{r\rightarrow0+}\frac{\log\mu(B(\xv,r))}{\log r} \quad \text{and} \quad \udimloc(\mu,\xv)=\limsup_{r\rightarrow0+}\frac{\log\mu(B(\xv,r))}{\log r}
\end{equation*}
the lower and upper local dimension of $\mu$ at the point $\xv$, respectively. If the limit exists, then we say that the measure has local dimension $\dimloc(\mu,\xv)$ at the point $\xv$. For a given Radon measure $\mu$, the local dimensions naturally introduce four different dimensions:
\begin{align*}
	\ldimh\mu&=\essinf_{\xv\sim\mu}\ldimloc(\mu,\xv)=\inf\left\{\dimh A:A\text{ is a Borel set with }\mu(A)>0\right\},\\
	\udimh\mu&=\esssup_{\xv\sim\mu}\ldimloc(\mu,\xv)=\inf\left\{\dimh A:A\text{ is a Borel set with }\mu(A^c)=0\right\},\\
	\ldimp\mu&=\essinf_{\xv\sim\mu}\udimloc(\mu,\xv)=\inf\left\{\dimp A:A\text{ is a Borel set with }\mu(A)>0\right\},\\
	\udimp\mu&=\esssup_{\xv\sim\mu}\udimloc(\mu,\xv)=\inf\left\{\dimp A:A\text{ is a Borel set with }\mu(A^c)=0\right\},
\end{align*}
where $A^c$ denotes the complement of the set $A$. For proofs of the above characterizations via set-dimensions, see \cite{Fb2}. Moreover, we call the measure $\mu$ exact dimensional if the local dimension exists for $\mu$-almost every point and equals to $$\ldimh \mu=\udimh \mu=\ldimp \mu=\udimp \mu.$$
In this case, the common value is denoted by $\dim\mu$.

The problem of the existence of local dimensions has a long history for self-affine measures and also in smooth dynamical systems. For an invariant measure in high-dimensional $C^{1+\alpha}$ systems, Ledrappier and Young~\cite{LY1,LY2} proved the existence of the local dimensions along stable and unstable local manifolds. Eckmann and Ruelle~\cite{ER} indicated that it is unknown whether the local dimension
of a hyperbolic invariant measure is the sum of the local dimensions along stable and unstable local manifolds. This question
was referred to as the Eckmann-Ruelle conjecture, and it was later confirmed by Barreira, Pesin, and Schmeling~\cite{BPS}. However, it remained open for non-smooth systems, such as self-affine measures.

Let $\Sigma$ be the set of one-sided words of symbols $\left\{1,\dots,N\right\}$ with infinite length, i.e.\ $\Sigma=\left\{1,\dots,N\right\}^{\N}$, where we adopt the convention that $0\in\N$. Let us denote the left-shift operator on $\Sigma$ by $\sigma$. Let the set of words with finite length be $\Sigma^*=\bigcup_{n=0}^{\infty}\left\{1,\dots,N\right\}^n$ and denote the length of $\iiv \in \Sigma^*$ by $|\iiv|$. We define the cylinder sets of $\Sigma$ in the usual way, that is, by setting
$$
[i_0,\dots,i_n]=\left\{\jj=(j_0,j_1,\dots)\in\Sigma:i_0=j_0,\dots,i_n=j_n\right\}.
$$
For a word $\iiv=(i_0,\dots,i_n)$ with finite length let $f_{\iiv}$ be the composition $f_{i_0}\circ\cdots\circ f_{i_n}$ and $A_{\iiv}$ be the product $A_{i_0}\cdots A_{i_n}$.

Let $\nu$ be a probability measure on $\Sigma$. We say that $\nu$ is \textit{quasi-Bernoulli} if there exists a constant $C\geq1$ such that for every $\iiv,\jjv\in\Sigma^*$
$$
C^{-1}\nu([\iiv])\nu([\jjv])\leq\nu([\iiv\jjv])\leq C\nu([\iiv])\nu([\jjv]).
$$
We note that this definition suffices to us since in the proofs we can always replace the quasi-Bernoulli measure $\nu$ by a $\sigma$-invariant ergodic quasi-Bernoulli measure equivalent to $\nu$. If the constant $C \geq 1$ above can be chosen to be $1$, then $\nu$ is called \emph{Bernoulli}. It is easy to see that a Bernoulli measure is $\sigma$-invariant and ergodic. By definition, for any Bernoulli measure $\nu$ there exists a probability vector $\pv=(p_1,\dots,p_N)$ such that $\nu([i_0,\dots,i_n])=p_{i_0}\cdots p_{i_n}$.

We define a \textit{natural projection} from $\Sigma$ to $\Lambda$ by $$
\pi(i_0i_1\cdots)=\lim_{n\rightarrow\infty}f_{i_0}\circ f_{i_1}\circ\cdots\circ f_{i_n}(\underline{0}),
$$
where $\underline{0}$ denotes the zero vector in $\R^d$. If $\nu$ is a Bernoulli measure, then the push-down measure $\mu=\pi_*\nu=\nu\circ\pi^{-1}$ is called \textit{self-affine}, and if $\nu$ is quasi-Bernoulli then $\mu$ is called \textit{quasi self-affine}. It is well known that a self-affine measure $\mu$ satisfies
\begin{equation} \label{eq:s-a_measure}
  \mu = \sum_{i=1}^N p_i \mu \circ f_i^{-1},
\end{equation}
where $(p_1,\ldots,p_N)$ is the associated probability vector.

If the linear parts of $f_i$ are similarities then we call the self-affine measure \emph{self-similar}. Ledrappier indicated that applying the method used in \cite{LY1, LY2}, one could prove exact dimensionality for self-similar measures; see Peres and Solomyak~\cite[p.\ 1619]{PeSo}. This was later conjectured by Fan, Lau, and Rao~\cite{FLR} and finally confirmed by Feng and Hu~\cite{FH}. We remark that Feng and Hu~\cite{FH} proved the result for the push-down measure of any ergodic $\sigma$-invariant measure.

The first result for self-affine systems is due to McMullen \cite{Mc}, who implicitly proved the exact dimensionality of self-affine measures on the Bedford-McMullen carpets. Later, Gatzouras and Lalley~\cite{GL} showed the exact dimensionality and calculating the value of dimension of self-affine measures for a class of planar carpet-like self-affine measures. In fact, their method to calculate the Hausdorff dimension of carpet-like self-affine sets was to find the maximal possible dimension of self-affine measures. Later Bara\'nski~\cite{B} showed similar result for another class of planar self-affine carpets. In addition to the self-similar case, Feng and Hu~\cite{FH} proved exact dimensionality for push-down measure of arbitrary ergodic $\sigma$-invariant measure on box-like self-affine sets. This was previously suggested by Kenyon and Peres~\cite{KP} for higher dimensional Bedford-McMullen carpets. If the strong separation condition holds and the linear parts satisfy the dominated splitting condition, B\'ar\'any~\cite{B} showed exact dimensionality of planar self-affine measures and gave the dimension a formula which involves entropy, Lyapunov exponents and projections of the measure. This formula was first shown by Ledrappier and Young~\cite{LY1,LY2} for the local dimension along stable manifolds of invariant measures of $C^2$ smooth diffeomorphisms.

K\"aenm\"aki~\cite{K} and K\"aenm\"aki and Vilppolainen~\cite{KV} showed that for almost every translation vector there exists an ergodic $\sigma$-invariant measure such that the upper Hausdorff dimension of the push down measure is equal to the affinity dimension and hence, is the maximum possible. For almost every positive matrices taken from certain open set, B\'ar\'any and Rams~\cite{BR} showed that this dimension maximizing measure exists and is exact dimensional for all translation vectors provided that the strong separation condition holds.

The main result, Theorem~\ref{texact2}, of this paper confirms that every self-affine measure is exact dimensional provided that the corresponding Lyapunov exponents are distinct. As a corollary, we solve a long standing open problem in the plane by showing that every planar self-affine measure is exact-dimensional regardless of the choice of matrices and translation vectors. This generalises the results of Gatzouras and Lalley~\cite{GL} and Feng and Hu~\cite{FH} in the plane. By introducing the projected entropy, exact dimensionality of self-similar measures was proven by Feng and Hu~\cite{FH}. Relying on the product structure of box-like self-affine systems, they were able to show the Ledrappier-Young formula in terms of the sequence of projected entropies. In our case, since the choice of matrices is free, we have the added complication coming from the non-existence of invariant directions. Therefore we adapt the original method of Ledrappier-Young~\cite{LY1,LY2} by considering orthogonal projections instead of locally defined invariant H\"older manifolds. The adaptation is not straightforward since the induced dynamical system has singularities and is not invertible.

In Theorem~\ref{texactd}, we prove that every quasi self-affine measure on $\R^d$ is exact dimensional if the corresponding matrices satisfies the totally dominated splitting condition. We show that the Ledrappier-Young formula holds also in this case.

Kaplan and Yorke conjectured that for Sinai-Ruelle-Bowen (SRB) measures the Hausdorff dimension is generically equal to the Lyapunov dimension; see Eckmann and Ruelle~\cite{ER}. The self-affine and quasi self-affine measures can be defined as SRB-measures of some Baker-transformation with singularities. Jordan, Pollicott, and Simon~\cite{JPS} showed that if the norm of the linear parts is less than $1/2$ then for Lebesgue almost every translation vector the lower and upper Hausdorff dimensions of the push-down measure of any ergodic $\sigma$-invariant measure coincide, and the value is equal to the Lyapunov dimension of the measure. Rossi~\cite{Ro} extended this result for packing dimensions. 

As a corollary to our results, we reformulate the Kaplan-Yorke conjecture for self-affine and quasi self-affine measures in Corollary~\ref{cor:KYd}.


\subsection*{Structure of the paper} In Section~\ref{sec:results}, we state our main results and exhibit a few corollaries. In Section~\ref{sec:condmeasure}, we give a general overview on conditional measures of Radon measures with respect to measurable partitions and prove a couple of auxiliary results. In Section~\ref{sec:Lift}, we introduce the dynamical system used to study self-affine measures on $\R^d$. We define the system in $\R^{d+1}\times\mathbb{F}$ by lifting the planar IFS into $\R^{d+1}$ such that it satisfies the strong separation condition. Here $\mathbb{F}$ is the set of flags and its role is to keep track of strong stable directions. Moreover, we also define families of invariant measurable partitions associated to the Lyapunov exponents/stable directions. In Section~\ref{sec:proofLY2}, we prove our main result on self-affine measures having distinct Lyapunov exponents, Theorem~\ref{texact2}. The proof is decomposed into three propositions. At first, we show that the conditional measures along the strongest stable directions are exact dimensional. Secondly, we prove that the projections along strong stable directions of conditional measures onto weaker stable directions are exact dimensional. Finally, we show that the conditional measures have product structure with respect to the strong stable foliations and projections along strong stable foliations.

We note that there is a remarkable difference between the case of general matrices with self-affine measures having distinct Lyapunov exponents and the case of matrices satisfying the totally dominated splitting condition with quasi self-affine measures. In the latter case, because of the result of Bochi and Gourmelon~\cite{BG}, we can define a H\"older continuous function from the symbolic space to the space of sequences of subspaces, which are the strong stable subspaces. Therefore in Section~\ref{sec:Liftd}, we can define our dynamical system on $\R^d\times\Sigma$. However, for general matrices such H\"older function does not necessarily exist. Therefore, we have to define our dynamical system on $\R^{d+1}\times\mathbb{F}$, which is clearly not invertible nor hyperbolic. This fact also restricts our analysis with general matrices by requiring the measure to be self-affine.

In Section~\ref{sec:proofLYd}, we prove the case of matrices satisfying the totally dominated splitting condition, Theorem~\ref{texactd}. To prove that the projections along strong stable directions of conditional measures onto weaker stable directions are exact dimensional is the main contribution of this section.

\section{Main results}\label{sec:results}

\subsection{Ledrappier-Young formula for Bernoulli measures with simple spectrum}
To state our first main theorem, let us recall here the statement of the Oseledets' Theorem for one-sided shifts; see \cite[Theorem~3.4.1]{A}.

	\begin{theorem}[Oseledets]\label{thm:Oseledets}
		Let $\mathcal{A}=\left\{A_1,\dots,A_N\right\}$ be a set of non-singular $d\times d$ matrices with $\|A_i\|<1$ for $i \in \{ 1,\dots,N \}$. Then for any ergodic $\sigma$-invariant measure $\nu$ on $\Sigma$ there exist constants $0<\ly{1}\leq\cdots\leq\ly{d}$ such that
		$$
		\lim_{n\to\infty}\frac{1}{n}\log\alpha_{i}(A_{i_n}^{-1}\cdots A_{i_0}^{-1})=\ly{d-i+1}
		$$
		for $\nu$-almost every $\ii$. There exist $p \in \{ 1,\ldots,d\}$ and $d_j\geq1$ for $j\in \{ 1,\dots,p \}$ such that
		$$
		\ly{1}=\dots=\ly{d_1}<\ly{d_1+1}=\dots=\ly{d_1+d_2}<\dots<\ly{d_1+\dots+d_{p-1}+1}=\dots=\ly{d}.
		$$
		Moreover, for every $j \in \{ 1,\dots,p \}$ and $\nu$-almost every $\ii\in\Sigma$ there exist a $d_1+\dots+d_j$-dimensional subspace $E^j(\ii)$ of $\R^d$ depending measurably on the point such that
		$$
		E^1(\ii)\subset \cdots \subset E^{p}(\ii), \quad A_{i_0}^{-1}E^j(\ii)=E^j(\sigma\ii), \quad \text{and} \quad \lim_{n\to\infty}\frac{1}{n}\log\|A_{i_n}^{-1}\cdots A_{i_0}^{-1}|E^j(\ii)\|=\ly{d_1+\dots+d_j}
		$$
		for $\nu$-almost every $\ii$. The numbers $\ly{i}$ are called the \emph{Lyapunov-exponents} of $\nu$.
	\end{theorem}
	
From the geometric point of view, the Oseledets theorem states that for large $n$, a typical linear map $\xv\mapsto A_{i_n}^{-1}\cdots A_{i_0}^{-1}\xv$ maps the unit ball to an ellipse with semiaxes of length approximately $e^{\ly{i} n}$. If two Lyapunov exponents coincide, then the ratio of the lengths of the corresponding semiaxes may still converge to zero, but subexponentially.
	We say that $\nu$ has \textit{simple Lyapunov spectrum} if all the Lyapunov exponents have multiplicity one, that is,
	\begin{equation}\label{eq:simplespect}
	\ly{1}<\ly{2}<\cdots<\ly{d}.
	\end{equation}
	Let us denote the Grassmannian manifold of $k$-dimensional subspaces of $\R^d$ by $G(k,d)$. Moreover, for every $1\leq j_1<\cdots<j_p\leq d$ let
	$$
	\mathbb{F}^d_{(j_1,\dots,j_p)}=\left\{V_p\times\cdots\times V_1\in G(j_1,d)\times\cdots\times G(j_p,d):V_p\subset\cdots\subset V_1\right\}
	$$
	be the space of flags.
	
	\begin{theorem}[Existence of Furstenberg measure]\label{thm:furst}
		Let $\mathcal{A}=(A_1,\dots,A_N)$ and $\nu$ be a Bernoulli measure on $\Sigma$. Moreover, let $\tau=(d_p,d_p+d_{p-1},\dots,d_p+\cdots+d_2)$, where $p$ and $d_i$ are as in Theorem~\ref{thm:Oseledets}. Define $T \colon \Sigma\times\mathbb{F}^d_{\tau}\mapsto\Sigma\times\mathbb{F}^d_{\tau}$ such that
		$$
		T(\ii,V_{p-1}\times\cdots\times V_1)=(\sigma\ii,A_{i_0}^{-1}V_{p-1}\times\cdots\times A_{i_0}^{-1}V_1).
		$$
		Then there exists a measure $\mu_F$ on $\mathbb{F}^d_{\tau}$ such that $\nu\times\mu_F$ is $T$-invariant and ergodic. Furthermore,
		$$
		\lim_{n\to\infty}\frac{1}{n}\log\mathfrak{m}(A_{i_n}^{-1}\cdots A_{i_0}^{-1}|V_j)=\ly{d_1+\dots+d_{j+1}}
		$$
		for $\nu\times\mu_F$-almost all $(\ii,V_{p-1}\times\cdots\times V_1)$.
	\end{theorem}
	
	
	\begin{proof}
		Let $\hat{\Sigma}=\{1,\dots,N\}^{\mathbb{Z}}$ and denote the unique two-sided extension of $\nu$ to $\hat{\Sigma}$ by $\mathbb{P}$. Moreover, let $\hat{T}\colon \hat{\Sigma}\times\mathbb{F}^d_{\tau} \mapsto\hat{\Sigma}\times\mathbb{F}^d_{\tau}$ be the two-sided extension of $T$. Let us denote by $\mathcal{F}$ the sigma-algebra generated by cylinder sets of $\hat{\Sigma}$. Denote by $\mathcal{F}^+$ (and respectively by $\mathcal{F}^-$)  the sub-sigma-algebras of $\mathcal{F}$ restricted to $\Sigma$ (and restricted to $\Sigma_-=\{1,\dots,N\}^{\mathbb{Z}_-}$).
		By the two-sided Oseledets's Theorem (see \cite[Theorem 3.4.11]{A}), for $\mathbb{P}$-almost every $\ii$ there exist $d_1,\dots,d_p$ dimensional subspaces $\hat{E}^1(\ii),\dots,\hat{E}^p(\ii)$ such that
		$$
                  \R^d=\hat{E}^1(\ii)\oplus\cdots\oplus\hat{E}^p(\ii),\quad A_{i_0}^{-1}\hat{E}^j(\ii)=\hat{E}^j(\sigma\ii),
		$$
		and
		$$\lim_{n\to\infty}\frac{1}{n}\log\|A_{i_n}^{-1}\cdots A_{i_0}^{-1}v\|=\ly{d_1+\cdots+d_j}
		$$
		for all $v\in \hat{E}^j(\ii)\setminus\{0\}$.
		By \cite[Remark~5.3.2]{A}, the function $\ii\mapsto \hat{E}^p(\ii)\oplus\cdots\oplus\hat{E}^j(\ii)$ is $\mathcal{F}^-$-measurable for all $j\in\{1,\dots,p\}$.
		By \cite[Theorem~1.6.13 and Theorem~1.7.5]{A}, there exists a $\hat{T}$-invariant and ergodic measure $\mu$ such that
		$$
		\mathrm{d}\mu(\ii,V_{p-1}\times\cdots\times V_1)=\mathrm{d}\delta_{\hat{E}^p(\ii)}\cdots \mathrm{d}\delta_{\hat{E}^p(\ii)\oplus\cdots\oplus\hat{E}^2(\ii)}\mathrm{d}\mathbb{P}(\ii),
		$$
		where $\delta_x$ denotes the Dirac-measure supported on $x$. Thus, by \cite[Theorem~1.7.2 and Corollary~1.7.6]{A}, there exists a measure $\mu_F$ on $\Sigma$ such that
		$$
		\mathbb{E}(\mu|\mathcal{F}^+)=\nu\times\mu_F
		$$
		and $\nu\times\mu_F$ is $T$-invariant. The last assertion of the theorem follows by the definition of the subspaces $\hat{E}^j$.
	\end{proof}

The measure $\mu_F$ is called the \emph{Furstenberg measure}. We note that the Furstenberg measure may not be unique, but our results are independent of the choice.  To simplify notation, we denote the elements of $\mathbb{F}^d_{1,\dots,d-1}$ by $\theta=(V_{d-1},\dots,V_1)$. The entropy of a quasi-Bernoulli measure $\nu$ is
$$
h_{\nu}=-\lim_{n\to\infty}\frac{1}{n}\sum_{|\iiv|=n}\nu([\iiv])\log\nu([\iiv]).
$$
Note that if $\nu$ is a Bernoulli measure obtained from a probability vector $(p_1,\dots,p_N)$, then $h_{\nu}=-\sum_{i=1}^Np_i\log p_i$.
Finally, for a subspace $V\subset\R^d$, we denote the orthogonal projection from $\R^d$ onto $V$ by $\proj_V$.

	\begin{theorem}[Main Theorem]\label{texact2}
		Let $\Phi=\left\{f_i(\xv)=A_i\xv+\tv_i\right\}_{i=1}^N$ be an IFS on $\R^d$ such that $\Alpha=\left\{A_1, \dots, A_N\right\}$ is a finite set of contractive non-singular $d\times d$ matrices. Then for every Bernoulli measure $\nu$ on $\Sigma$ with simple Lyapunov spectrum, the self-affine measure $\mu=\pi_*\nu$ is exact dimensional. Moreover, $\mu^T_{\theta^{\bot}}=(\proj_{\theta^{\bot}})_*\mu$ is exact dimensional and
		\begin{equation} \label{etoprojd}
			\dim\mu=\frac{h_{\nu}-H}{\ly{d}}+\sum_{i=1}^{d-1}\left(\frac{\ly{i+1}-\ly{i}}{\ly{d}}\right)\dim\mu^T_{V_i^{\bot}}
		\end{equation}
		for $\mu_F$-almost every $(V_{d-1},\dots,V_1)$, where $H=-\int\log\nu^{\pi^{-1}}_{\ii}([i_0])\dd\nu(\ii)$ and $\{\nu^{\pi^{-1}}_{\ii}\}$ is the family of conditional measures of $\nu$ defined by the measurable partition $\{\pi^{-1}(\pi(\ii))\}$.
	\end{theorem}

The equation \eqref{etoprojd} is called the \emph{Ledrappier-Young formula}. The following theorem solves the long standing open problem on the exact dimensionality of planar self-affine measures.

\begin{theorem}\label{cor:planar}
	If $\Phi=\left\{f_i(\xv)=A_i\xv+\tv_i\right\}_{i=1}^N$ is an IFS on $\R^2$ such that $\Alpha=\left\{A_1, \dots, A_N\right\}$ is a finite set of contractive non-singular $2\times 2$ matrices then for every Bernoulli measure $\nu$, the self-affine measure $\mu=\pi_*\nu$ is exact-dimensional and satisfies the Ledrappier-Young formula.
\end{theorem}

\begin{proof}
	By Feng and Hu~\cite{FH}, if $\ly{1}=\ly{2}$, then $\mu$ is exact dimensional and $\dim\mu = (h_\nu-H)/\ly{1}$. The case $\ly{1}<\ly{2}$ follows from Theorem~\ref{texact2}.
\end{proof}

Let us compare Theorem \ref{cor:planar} with the existing planar results. Besides solving the problem in the case of equal Lyapunov exponents, Feng and Hu \cite[Theorem 2.11]{FH} proved the Ledrappier-Young formula for matrix tuples of diagonal matrices. B\'ar\'any \cite[Theorem~2.7]{Ba} proved the formula for matrix tuples satisfying the dominated splitting condition by assuming the strong separation condition. In Theorem~\ref{cor:planar}, we do not assume any kind of separation condition. The statement holds for any contracting matrices, regardless of overlaps.

The main idea of the proof is to show that the self-affine measure has a conformal structure restricted to the Oseledets and Furstenberg directions. This is done in Propositions \ref{prop:1} and \ref{pLY2}. By using this observation, we show in Propositions \ref{prop:LY3} and \ref{pLY4} that the original measure has a local product-like structure with respect to these restrictions and hence, the exact-dimensionality and the Ledrappier-Young formula follow. We overcome the problems coming from the lack of separation conditions by lifting the system in one dimension higher; this is done in Section \ref{sec:Lift}.

\subsection{Ledrappier-Young formula with totally dominated splitting matrix tuples}

Before we state our second main theorem we introduce the totally dominated splitting condition for a finite family of matrices.
        We say that a finite set of contractive non-singular $d\times d$ matrices $\Alpha=\left\{A_1, \dots, A_N\right\}$ has \emph{dominated splitting} of index $i \in \{ 1,\ldots,d-1 \}$ if there exist constants $C\geq1$ and $0<\tau<1$ such that
        $$
        \frac{\alpha_{i+1}(A_{j_0}\cdots A_{j_n})}{\alpha_{i}(A_{j_0}\cdots A_{j_n})}\leq C\tau^n
        $$
        for all $j_0,\dots,j_n\in\left\{1,\dots,N\right\}$ and $n \in \N$.        Furthermore, we say that $\Alpha$ satisfies the \emph{totally dominated splitting} (TDS) if for every $i\in \{ 1,\dots,d-1 \}$ either $\mathcal{A}$ has dominated splitting of index $i$ or there exists a constant $C\geq1$ such that
        $$
        C^{-1}\leq\frac{\alpha_{i+1}(A_{j_0}\cdots A_{j_n})}{\alpha_{i}(A_{j_0}\cdots A_{j_n})}
        $$
	for all $j_0,\dots,j_n\in\left\{1,\dots,N\right\}$ and $n \in \N$.
	We call the set of indices, where $\alpha_{i}$ dominates $\alpha_{i+1}$, \emph{dominated indices} and we denote it by $\mathcal{D}(\mathcal{A})$. In other words, $\mathcal{D}(\mathcal{A}) = \{ i \in \{ 1,\ldots,d-1 \} : \mathcal{A}$ has dominated splitting of index  $i \}$.

By Oseledec's Multiplicative Ergodic Theorem, for any $\sigma$-invariant ergodic measure $\nu$ there are constants $0<\chi^1_{\nu}\leq\cdots\leq\chi^d_{\nu}$ such that
$$
\lim_{n\rightarrow\infty}\frac{1}{n}\log\alpha_i(A_{i_0}\cdots A_{i_n})=-\chi^i_{\nu}
$$
for $\nu$-almost every $\ii\in\Sigma$.
In particular, if $\mathcal{A}$ satisfies the TDS, then $\chi^i_{\nu}<\chi^{i}_{\nu}-\log\tau\leq\lye{\nu}{i+1}$ for every $i\in\mathcal{D}(\mathcal{A})$ and $\lye{\nu}{i}=\lye{\nu}{i+1}$ for $i\notin\mathcal{D}(\mathcal{A})$.

\begin{prop}[Bochi and Gourmelon \cite{BG}] \label{thm:BGprop}
	Let $\Alpha=\left\{A_1, \dots, A_N\right\}$ be a finite set of contractive non-singular $d \times d$ matrices satisfying the TDS. Then for every $i\in\mathcal{D}(\mathcal{A})$ there exists a family of subspaces $\left\{F_{\ii}^i\right\}_{\ii\in\Sigma}$ such that
	\begin{itemize}
		\item[(1)] $\dim F_{\ii}^i=d-i$,
		\item[(2)] $A_{i_0}F_{\ii}^i=F_{\sigma\ii}^i$,
		\item[(3)] $\|A_{i_n}\cdots A_{i_0}|F_{\ii}^i\|\leq C\alpha_{i+1}(A_{i_n}\cdots A_{i_0})$,
		\item[(4)] the mapping $\ii \mapsto F_\ii^i$, denoted by $F^i$, is H\"older continuous.
	\end{itemize}
	Moreover, if the elements of $\mathcal{D}(\mathcal{A})$ are $1\leq j_1<\cdots < j_k\leq d-1$, then we have $F_{\ii}^{j_k}\subset F_{\ii}^{j_{k-1}}\subset\cdots\subset F^{j_1}_{\ii}$.
\end{prop}

In the case totally dominated splitting, we can define the Furstenberg measure $\mu_F$ on $\mathbb{F}_{d-j_k,\dots,d-j_1}^d$ by
$
\mu_F=(F_{\cdot}^{j_k},\dots,F_{\cdot}^{j_1})_*\nu.
$
Let us note that, by definition, $\dim\mu^T_{(F_{\ii}^i)^{\bot}}\leq\min\left\{i,\dim\mu\right\}$ for every $\ii\in\Sigma$ and $i\in\mathcal{D}(\Alpha)$.

\begin{theorem}[Main Theorem]\label{texactd}
	Let $\Phi=\left\{f_i(\xv)=A_i\xv+\tv_i\right\}_{i=1}^N$ be an IFS on $\R^d$ such that $\Alpha=\left\{A_1,\dots, A_N\right\}$ is a finite set of contractive non-singular $d\times d$ matrices satisfying the TDS. Then for every quasi-Bernoulli measure $\nu$ on $\Sigma$, the quasi self-affine measure $\mu=\pi_*\nu$ is exact dimensional and for each $i\in\mathcal{D}$, the measure $\mu^T_{(F^{i}_{\ii})^{\bot}}=(\proj_{(F^{i}_{\ii})^{\bot}})_*\mu$ is exact dimensional for $\nu$-almost every $\ii$. Moreover,
	\begin{equation*}
	\dim\mu=\frac{h_{\nu}-H}{\ly{d}}+\sum_{i\in\mathcal{D}}\left(\frac{\ly{i+1}-\ly{i}}{\ly{d}}\right)\dim\mu^T_{(F^{i}_{\ii})^{\bot}}
	\end{equation*}
	for $\nu$-almost every $\ii$, where $H=-\int\log\nu^{\pi^{-1}}_{\ii}([i_{0}])\dd\nu(\ii)$ and $\{\nu^{\pi^{-1}}_{\ii}\}$ is the family of conditional measures of $\nu$ defined by the measurable partition $\{\pi^{-1}(\pi(\ii))\}$.
\end{theorem}

The main idea of the proof is essentially the same as that of Theorem \ref{texact2}. The main difference is in the verification of the conformal structure in the Oseledets and Furstenberg directions.

\subsection{Corollaries to the main theorems}
As a direct corollary to our results, under the respective assumptions, we can give a reformulation of the Kaplan-Yorke conjecture. This reformulation gives another perspective to verify or to disprove the conjecture.


\begin{cor}\label{cor:KYd}
	Under the assumptions of Theorem~\ref{texact2} or Theorem~\ref{texactd},
	\begin{equation*}
		\dim\mu=\min_{k\in\{1,\dots,d\}}\biggl\{k-1+\frac{h_{\nu}-\sum_{i=1}^{k-1}\ly{i}}{\ly{k}}\biggr\}
	\end{equation*}
	holds if and only if
	\begin{equation*}
		H=0\text{ and }\dim\mu^T_{V_i^{\bot}}=\min\left\{i,\dim\mu\right\}
	\end{equation*} for every $i\in\mathcal{D}$ and $\mu_F$-almost every $\theta$, where either $\mathcal{D}=\{1,\dots,d-1\}$ or $\mathcal{D}$ is the set of dominated indexes.
\end{cor}

In the view of Corollary \ref{cor:KYd}, one may expect that there is an equivalent characterisation of the Kaplan-Yorke conjecture for ergodic invariant measures of hyperbolic diffeomorphisms acting on smooth Riemannian manifolds. This would mean that the Kaplan-Yorke conjecture holds if and only if there is no dimension drop for typical projections along the tangent bundles of $C^{1+\alpha}$ stable and unstable leafs.

For planar self-affine measures, Falconer and Kempton \cite{FK2} 
have recently shown that if the projected measures $\mu^T_{\theta^{\bot}}$ are exact dimensional for $\mu_F$-almost every $\theta$, then the $\mu_F$-typical value of the dimension of $\mu^T_{\theta^{\bot}}$ is minimal except possibly for at most one direction.

We say that $\Phi$ satisfies the \textit{strong open set condition} (SOSC) if there exists an open and bounded set $U\subset\R^d$ intersecting the self-affine set, $U \cap \Lambda \ne \emptyset$, such that
$$
\bigcup_{i=1}^Nf_i(U)\subseteq U
$$
and $f_i(U)\cap f_j(U)=\emptyset$ for every $i\neq j$.
The SOSC is a milder condition than \textit{strong separation condition} (SSC), which holds if $f_i(\Lambda)\cap f_j(\Lambda)=\emptyset$ for $i\neq j$.

\begin{cor}\label{c:LY}
  Under the assumptions of Theorem~\ref{texact2} or Theorem~\ref{texactd}, if $\Phi$ satisfies the SOSC and $\nu([i])>0$ for all $i\in \{1,\dots,N \}$, then
  \begin{equation*}
    H=-\int\log\nu^{\pi^{-1}}_{\ii}([i_{0}])\dd\nu(\ii)=0.
  \end{equation*}
\end{cor}

\begin{proof}
	It is enough to show that $\pi^{-1}(\pi(\ii))$ is a singleton for $\nu$-almost every $\ii$. Let us define two sets,
	$$
	I=\left\{\ii\in\Sigma:\pi(\ii)\in U\right\}\quad \text{and} \quad C=\left\{\ii\in\Sigma:\text{ there exist }\jj\neq\ii\text{ such that }\pi(\ii)=\pi(\jj)\right\},
	$$
	where $U$ is the open set of the SOSC.
	It is easy to see that $\sigma^{-1}I\subseteq I$, therefore by ergodicity either $\nu(I)=0$ or $\nu(I)=1$. Since $U\cap\Lambda\neq\emptyset$ there exists a cylinder set $f_{\iiv}(\Lambda)$ that $f_{\iiv}(\Lambda)\subset U$. Hence, $\nu(I)=\mu(U)\geq\mu(f_{\iiv}(\Lambda))\geq\nu([\iiv])>0$ and therefore $\mu(U)=\nu(I)=1$, i.e. $\mu(\partial U)=0$.
	
	On the other hand, if $\ii\in C$ then there exists $\jj\neq\ii$ such that $\pi(\ii)=\pi(\jj)$. Let $n=\min\left\{k:i_k\neq j_k\right\}$. Then $\pi(\sigma^{n}\ii)=\pi(\sigma^{n}\jj)$ but the first symbols of $\sigma^{n}\ii$ and $\sigma^{n}\jj$ differ. Since $\Lambda$ is contained in the closure of $U$ and $f_i(U)\cap f_j(U)=\emptyset$ for every $i\neq j$ it is only possible if $\pi(\sigma^{n}\ii)\in \partial f_{i_n}(U)$. Thus, $\pi(\sigma^{n+1}\ii)\in\partial U$ and therefore $\pi(C)\subseteq\bigcup_{n=0}^{\infty}\bigcup_{|\iiv|=n}f_{\iiv}(\partial U)$.
	Hence, by \eqref{eq:s-a_measure},
	$$
	\nu(C)\leq\mu(\pi(C))\leq\mu\biggl(\bigcup_{n=0}^{\infty}\bigcup_{|\iiv|=n}f_{\iiv}(\partial U)\biggr)\leq\sum_{n=0}^{\infty}\sum_{|\iiv|=n}\mu(f_{\iiv}(\partial U))=\sum_{n=0}^{\infty}\mu(\partial U)=0,
	$$
	which is what we wanted to show.
\end{proof}

As an easy consequence of Corollary~\ref{cor:KYd}, Corollary~\ref{c:LY}, and \cite[Theorem~2.6]{PS}, we get the following.

\begin{cor}
	Let $\Phi=\left\{f_i(\xv)=A_i\xv+\tv_i\right\}_{i=1}^N$ be an IFS on $\R^2$ such that $\Alpha=\left\{A_1, \dots, A_N\right\}$ is a finite set of contractive non-singular $2\times 2$ matrices. Let us assume that $\Phi$ satisfies the SOSC. Then for every Bernoulli measure $\nu$, if $\ldimh\mu_F\geq\min\left\{\dim\mu,2-\dim\mu\right\}$, then we have
	\begin{equation*}
		\dim\mu=\min\left\{\frac{h_{\nu}}{\chi^{1}_{\nu}},1+\frac{h_{\nu}-\chi^{1}_{\nu}}{\chi^{2}_{\nu}}\right\}.
	\end{equation*}
\end{cor}

We remark that Hochman and Solomyak \cite{HS} have recently announced a method to calculate the dimension of the Furstenberg measure $\mu_F$ for $2\times 2$ matrices. Many of the recent works on dimensions of self-affine sets rely on properties of the Furstenberg measure; see e.g.\ Morris and Shmerkin \cite{MS} and Rapaport \cite{Rapa}.

\begin{cor}
	Let $\Phi=\left\{f_i(\xv)=A_i\xv+\tv_i\right\}_{i=1}^N$ be an IFS on $\R^d$ such that $\Alpha=\left\{A_1,\dots, A_N\right\}$ is a finite set of contractive non-singular $d\times d$ matrices satisfying the TDS. Moreover, let us assume that $\Phi$ satisfies the SOSC. Then for every quasi-Bernoulli measure $\nu$ , if
	\begin{equation*}
	\ldimh(F^i)_*\nu+\dim\mu > i(d-i+1)\quad\text{or}\quad
	\ldimh(F^i)_*\nu \geq i(d-i-1)+\min\left\{i,\dim\mu\right\}
	\end{equation*}
	for every $i\in\mathcal{D}(\Alpha)$, then we have
	\begin{equation}\label{eq:lyapdim}
	\dim\mu=\min_{k\in \{1,\dots,d\}}\biggl\{k-1+\frac{h_{\nu}-\sum_{i=1}^{k-1}\ly{i}}{\ly{k}}\biggr\}.
	\end{equation}
\end{cor}

\begin{proof}[Proof.]
Let $i \in \{1\ldots d\}$ and let $\lambda$ be a Radon measure with finite $t$-energy with $t<i$. Then there exists a set $X_t\subset G(i,d)$ with $\dimh X_t \leq i(d-i-1)+t$ such that for all $V\in G(i,d)\setminus X_t$, the measure $(\proj_V)_*\lambda$ has finite $t$-energy; see \cite[Theorem~2.2(i)]{FaMa}.
Observe that for every $i\in\mathcal{D}$, we have $\ldimh((F^i)^{\bot})_*\nu=\ldimh(F^i)_*\nu$.

Let us first assume that $\dim\mu\leq i$. By Egorov's Theorem and the exact dimensionality of $\mu=\pi_*\nu$ and $\mu^T_{(F^i_{\ii})^{\bot}}$, for every $\varepsilon>0$ there exists a set $E$ with $\mu(E)>1-\varepsilon$ such that $\mu|_E$ has finite $(\dim\mu-\varepsilon)$-energy. By choosing $\varepsilon>0$ small enough and $\lambda=\mu|_E$ we get that if
$$\ldimh(F^i)_*\nu > i(d-i-1)+\dim\mu-\varepsilon$$ then for
$\nu$-almost every $\ii$
$$
\dim\mu^T_{(F_{\ii}^i)^{\bot}}\geq\udimh\left(\mu|_E\right)^T_{(F_{\ii}^i)^{\bot}}\geq\ldimh\left(\mu|_E\right)^T_{(F_{\ii}^i)^{\bot}}\geq\dim\mu-\varepsilon.
$$
Since  $\varepsilon>0$ was arbitrary, we get that if $\dim\mu\leq i$, then $\dim\mu^T_{(F_{\ii}^i)^{\bot}}=\dim\mu$ for $\nu$-almost every $\ii$.

If $\dim\mu>i$, then by Egorov's Theorem choosing $\varepsilon>0$ sufficiently small, $\dim\mu|_E>i$ and $\mu|_E$ has finite $(\dim\mu-\varepsilon)$-energy. Thus, by \cite[Proposition~6.1]{PS} we get that if
$$
\ldimh(F^i)_*\nu+\dim\mu-\varepsilon> i(d-i+1),
$$
then $\dim\mu^T_{(F_{\ii}^i)^{\bot}}\geq\dim\left(\mu|_E\right)^T_{(F_{\ii}^i)^{\bot}}=i$ for $\nu$-almost every $\ii$. Therefore for every $i\in\mathcal{D}$
$$\dim\mu^T_{(F_{\ii}^i)^{\bot}}=\min\left\{i,\dim\mu\right\}$$
for $\nu$-almost every $\ii$.
By simple algebra, we see that Theorem~\ref{texactd} and Corollary~\ref{c:LY} imply \eqref{eq:lyapdim}.
\end{proof}

To finish this section, we exhibit a concrete example of a family of matrices satisfying the TDS. Let us recall the definition of totally postive matrices from \cite[Definition~1.1]{Pi}. Let $I_p^d=\left\{\iiv=(i_1,\dots,i_p):1\leq i_1<\cdots<i_p\leq d\right\}$ for $p \in \{1,\ldots,d\}$, and for any $\iiv,\jjv\in I_p^d$ let
$$
A[\iiv,\jjv]:=\det(a_{i_k,j_l})_{k,l=1}^p,
$$
where $A=(a_{i,j})_{i,j=1}^d$. Thus $A[\iiv,\jjv]$ is the minor of $A$ determined by $\iiv,\jjv$. We say that a $d \times d$ matrix $A$ is \textit{strictly totally positive} (STP) if $A[\iiv,\jjv]>0$ for all $\iiv,\jjv\in I_p^d$ and $p \in \{1,\ldots, d-1\}$. By definition, the set of matrix tuples formed by the STP matrices is an open subset of the set of all matrix tuples.
For example, in the two dimensional case, the STP matrices are the matrices with strictly positive elements and positive determinant. These matrices map the first quadrant of the plane strictly into itself.

\begin{example} \label{ex:STP}
	A finite set of contractive STP matrices satisfies the TDS with $\mathcal{D}=\left\{1,\dots,d-1\right\}$.
\end{example}

Before we verify this claim, we recall another result of Bochi and Gourmelon.

\begin{theorem}[Bochi and Gourmelon \cite{BG}]\label{tnsconddomsplit}
	A finite set $\Alpha=\left\{A_1, \dots, A_N\right\}$ of contractive non-singular $d\times d$ matrices has the dominated splitting of index $i$ if and only if there exists a non-empty proper subset $C\subsetneq G(p,d)$ that is strictly invariant under $\mathcal{A}$, i.e.\ $A_{i}\overline{C}\subset C^o$ for all $i \in \{1,\ldots,N\}$, and there is a $(d-p)$-plane that is transverse to all elements of $C$.
	\end{theorem}
	
In the two dimensional case, the set $C$ in Theorem \ref{tnsconddomsplit} is a finite union of closed cones. Since $C$ is mapped strictly into itself by all the matrices, we have a compact set of subspaces $X \subset C$, which is invariant under the action of linear maps, i.e.\ $X = \bigcup_iA_i X$, and has uniformly positive angle with the boundary of $C$. Similarly, the closure of the complement of the multicone $C$ is also mapped into itself by the inverses of the linear maps, and the invariant subset again has a uniformly positive angle with the boundary of $C$. Hence, every ellipse become narrower and thicker, with some uniform multiple, under the action of the linear maps. This is what the definition of the dominated splitting calls for.

		Let us next verify the claim in Example \ref{ex:STP}. The $p$th exterior power of $\R^d$ is denoted by $\wedge^p\R^d$. Then $\left\{e_{i_1}\wedge\dots\wedge e_{i_p}:(i_1,\dots,i_p)\in I_p^d\right\}$ forms a basis of the vector space $\wedge^p\R^d$, where $\left\{e_i\right\}_{i=1}^d$ is the standard orthogonal basis of $\R^d$. Let $A^{[p]}=(A[\iiv,\jjv])_{\iiv,\jjv\in I_p^d}$ for every $d\times d$ matrix $A$. Thus, $A^{[p]}$ defines a linear mapping on $\wedge^p\R^d$ such that
		$$
		A\tv_1\wedge\dots\wedge A\tv_p=A^{[p]}\tv_1\wedge\dots\wedge \tv_p
		$$
		for all $\tv_1,\dots,\tv_p\in\R^d$.
                For each $\iiv=(i_1,\dots,i_p)\in I_p^d$ we define $e_{\iiv}=e_{i_1}\wedge\dots\wedge e_{i_p}$.
		Let $\widehat{C}_p$ be a subset of $\wedge^p\R^d$ such that
		$$
	\widehat{C}_p=\{\tv_1\wedge\dots\wedge\tv_p=(-1)^k\sum_{\iiv\in I_p^d}\lambda_{\iiv}e_{\iiv} : k\in\{0,1\}\text{ and }\lambda_{\iiv}>0\text{ for every }\iiv\in I_p^d\}.
		$$
		Observe that the mapping $P_p \colon \widehat{C}_p\to G(p,d)$ defined by $P(\tv_1\wedge\dots\wedge\tv_p)=\langle\tv_1,\dots,\tv_p\rangle$ is a continuous embedding. Defining $C_p=P(\widehat{C}_p)$ we see that $C_p$ is an open subset of $G(p,d)$. Therefore, by showing that $\bigcup_{i=1}^N A_i\overline{C_p} \subset C_p$ and that there exists a $(d-p)$-plane transverse to all elements of $C_p$, Theorem \ref{tnsconddomsplit} verifies the claim in Example \ref{ex:STP}.
		
			
Let $V\in A_i\overline{C_p}$ for some $i$. Then there exist vectors $\tv_1,\dots,\tv_p$ such that $V=\langle A_i\tv_1,\dots,A_i\tv_p\rangle$ and
$$
\tv_1\wedge\dots\wedge\tv_p=(-1)^k\sum_{\iiv\in I_p^d}\lambda_{\iiv}e_{\iiv},\text{ where }\lambda_{\iiv}\geq0\text{ but there exists }\iiv\in I_{p}^d\text{ such that }\lambda_{\iiv}>0.
$$
Since
$$
A_i\tv_1\wedge\dots\wedge A_i\tv_p=A_i^{[p]}\tv_1\wedge\dots\wedge\tv_p=(-1)^k\sum_{\iiv\in I_p^d}\biggl(\sum_{\jjv\in I_p^d}A_i[\iiv,\jjv]\lambda_{\jjv}\biggr)e_{\iiv}
$$
we see that $\sum_{\jjv\in I_p^d}A_i[\iiv,\jjv]\lambda_{\jjv}>0$. Therefore, $A_i\tv_1\wedge\dots\wedge A_i\tv_p\in\widehat{C}_p$ and $V \in C_p$.

To show that there exists a $(d-p)$-plane that is transverse to $C_p$ it is enough to show that there exist vectors $\av_1,\dots,\av_{d-p}$ that $\tv_1\wedge\dots\wedge\tv_p\wedge\av_1\wedge\dots\wedge\av_{d-p}\neq 0$ for every $\tv_1\wedge\dots\wedge\tv_p\in\widehat{C}_p$. But this follows immediately since, by choosing $\av_1=e_{p+1},\dots,\av_{d-p}=e_d$, we have
$$
\tv_1\wedge\dots\wedge\tv_p\wedge\av_1\wedge\dots\wedge\av_{d-p}=(-1)^k\lambda_{1\dots p}(e_1\wedge\dots\wedge e_d)\neq0.
$$
We have now verified the claim in Example \ref{ex:STP}.

\section{Conditional measures}\label{sec:condmeasure}

Let $(\Omega,\mathcal{B},\lambda)$ be a probability space. If $\zeta$ is a measurable partition of $\Omega$, then by the result of Rokhlin \cite{R}, there exists a canonical system of conditional measures. That is, for $\lambda$-almost every $x\in\Omega$ there exists a measure $\lambda^{\zeta}_{x}$ supported on $\zeta(x)$, where $\zeta(x)$ is the partition element which contains $x$, such that for every measurable set $A\subseteq\Omega$ the function $x\mapsto\lambda_{x}^{\zeta}(A)$ is $\mathcal{B}_{\zeta}$-measurable, where $\mathcal{B}_{\zeta}$ is the sub-$\sigma$-algebra of $\mathcal{B}$ whose elements are union of the elements of $\zeta$, and
\begin{equation*}
\lambda(A)=\int\lambda_{x}^{\zeta}(A) \dd\lambda(x).
\end{equation*}
The conditional measures are uniquely defined up to a set of zero measure.

For two measurable partitions $\zeta_1$ and $\zeta_2$ we define the common refinement $\zeta_1\vee\zeta_2$ such that for every $x$, $(\zeta_1\vee\zeta_2)(x)=\zeta_1(x)\cap\zeta_2(x)$. Moreover, let us define the image of the partition $\zeta$ under of measurable function $g\colon\Omega\mapsto\Omega'$ in the natural way, i.e. by setting $(g\zeta)(x)=g^{-1}(\zeta(g(x)))$ for all $x$. We say that $\zeta_1$ is a refinement of $\zeta_2$ if for every $x$, $\zeta_1(x)\subseteq\zeta_2(x)$, and we denote it by $\zeta_1>\zeta_2$.

\begin{lemma}\label{lem:condtech}
	Let $(\Omega,\mathcal{B},\lambda)$ and $(\Omega',\mathcal{B}',\lambda')$ be probability spaces and $\zeta$ be a measurable partition on $\Omega$. Let $g:\Omega\mapsto\Omega'$ be measurable, bijective mapping such that $g^{-1}\zeta$ is measurable. Then
	$$
	\left(g_*\lambda\right)^{g^{-1}\zeta}_y=g_*(\lambda_{g^{-1}(y)}^{\zeta})
	$$
	for $g_*\lambda$-almost every $y$.
\end{lemma}

\begin{proof}
	By the definition of conditional measures,
	$$ \int\left(g_*\lambda\right)^{g^{-1}\zeta}_y\dd g_*\lambda(y)=g_*\lambda=\int g_*(\lambda_x^{\zeta})\dd\lambda(x)=\int g_*(\lambda_{g^{-1}(y)}^{\zeta})\dd g_*\lambda(y).
	$$
	Since $g_*(\lambda_{g^{-1}(y)}^{\zeta})$ is supported on $g(\zeta(g^{-1}(y)))=(g^{-1}\zeta)(y)$, the statement follows by the uniqueness of conditional measures.
\end{proof}

Observe that if $\zeta_1$ and $\zeta_2$ are two measurable partitions of $\Omega$ then
\begin{equation}\label{econdcond}
(\lambda_{x}^{\zeta_1})_{x}^{\zeta_2}=\lambda_{x}^{\zeta_1\vee\zeta_2}
\end{equation}
for $\lambda$-almost every $x$. Let us define the conditional entropy of a countable measurable partition $\zeta_1$ with respect to a measurable partition $\zeta_2$ in the usual way
$$
H(\zeta_1|\zeta_2)=-\int\log\lambda_{x}^{\zeta_2}(\zeta_1(x))\dd\lambda(x).
$$
For a subspace $V\subseteq\R^d$ let us define the transversal ball centred at $x\in \R^d$ with radius $r$ in the usual way, i.e.
\begin{equation*}
B^T_V(x,r)=\proj_V^{-1}B_V(\proj_V(x),r),
\end{equation*}
where $B_V(x,r)$ denotes the Euclidean ball centred at $x$ with radius $r$ on $V$. By \cite[Theorem~2.2]{Si}, for the measurable partition $\zeta(x)=\proj_V^{-1}(\proj_V(x))$,
\begin{equation}\label{eq:Simmons}
	\lambda^{\zeta}_x=\lim_{r\to0+}\frac{\lambda|_{B^T_V(x,r)}}{\lambda(B^T_V(x,r))}
\end{equation}
for $\lambda$-almost every $x$.

\begin{lemma}\label{llbmuubtmu}
	Let $\lambda$ be a compactly supported Radon measure on $\R^d$ and $V$ a subspace of $\R^d$. Let $\zeta(x)=\proj_V^{-1}(\proj_V(x))$. If $\ldimloc(\lambda^{\zeta}_{x},x)\geq\alpha$ for $\lambda$-almost all $x$, then
	$$
	\ldimloc(\lambda,x)\geq\alpha+\ldimloc(\lambda^T_V,x) \quad \text{and} \quad \udimloc(\lambda,x)\geq\alpha+\udimloc(\lambda^T_V,x)
	$$
	for $\lambda$-almost every $x$, where $\lambda^T_V=(\proj_V)_*\lambda$.
\end{lemma}

The proof of the lemma can be found in \cite[Lemma~11.3.1]{LY2}.

\section{Lifted dynamical system and partitions}\label{sec:Lift}

In this section, we introduce a new dynamical system which helps us to overcome the issues caused by the lack of separation conditions. The partitions and conditional measures we utilize are natural with respect to this dynamical system. In forthcoming sections, we prove that on almost every such partition element the measure is exact-dimensional and has a local product-like structure formed by conformal measures.

Throughout this section we always assume that $\nu$ is a Bernoulli measure with simple Lyapunov spectrum.
Let $\Phi = \{ f_i(\xv) = A_i\xv + \tv_i \}_{i=1}^N$ be an IFS on $\R^d$ such that $\mathcal{A}=\{ A_1,\ldots,A_N \}$ is a finite set of contractive non-singular $d \times d$ matrices. Choose $0<\rho<\min\left\{1/N,\min_i\alpha_d(A_i)\right\}$ and let $\widehat{\Alpha}=\{\widehat{A}_1,\dots,\widehat{A}_N\}$ be the set of contractive non-singular $(d+1)\times(d+1)$ matrices such that
$$
\widehat{A}_i=
\begin{pmatrix}
A_i & \underline{0} \\
\underline{0}^T & \rho
\end{pmatrix}.
$$
Because of the definition of $\rho$, we can choose $\tau_1,\dots,\tau_N\in[0,1]$ to be real numbers such that the IFS $\widehat{\Phi}=\{\widehat{f}_i(\xv)=\widehat{A}_i\xv+\underline{\tau}_i\}_{i=1}^N$ satisfies the SSC, where $\underline{\tau}_i=(\tv_i,\tau_i)$. Let $\widehat{\Lambda}$ be the self-affine set associated to $\widehat{\Phi}$. Denote by $\widehat{\pi}$ the natural projection from $\Sigma$ to $\widehat{\Lambda}$, with respect to the IFS $\widehat{\Phi}$.
For simplicity, let us denote the space of flags $\mathbb{F}^d_{(d-1,\dots,1)}$ by $\mathbb{F}$. Recall that the elements of $\mathbb{F}$ are denoted by $\theta$.

Let $G\colon\widehat{\Lambda}\times\mathbb{F}\to\widehat{\Lambda}\times\mathbb{F}$ be such that
$$
G(\widehat{\pi}(\ii),\theta)=(\widehat{\pi}(\sigma\ii),A^{-1}_{i_0}\theta),
$$
where $\ii=(i_0i_1\cdots)$. To simplify notation, we often write $x=(\widehat{\pi}(\ii),\theta)$ and $\Omega=\widehat{\Lambda}\times\mathbb{F}$. By \eqref{eq:simplespect} and Theorem~\ref{thm:furst}, there exists a measure $\mu_F$ on $\mathbb{F}$ such that $\hat{\pi}_*\nu\times\mu_F$ is $G$-invariant and ergodic. We denote the measure $\widehat{\pi}_*\nu\times\mu_F$ by $\lambda$. Let $\widehat{\mu}=\widehat{\pi}_*\nu$ be the self-affine measure on $\widehat{\Lambda}$.

\begin{lemma}\label{lem:angle&dim}
	If \eqref{eq:simplespect} holds, then for $\nu$-almost every $\ii$ and $\mu_F$-almost every $(V_1,\dots,V_{d-1})$,
	$$
	\dim E^i(\ii)=i,\quad \dim V_j=d-j,\quad \text{and} \quad
	\dim E^i(\ii)\cap V_j=\max\{0,i-j\}
	$$
	for all $i,j\in\{1,\dots,d-1\}$.
	Moreover,
	\begin{equation}\label{eq:Furstenberg}
	\lim_{n\to\infty}\frac{1}{n}\log\|A_{i_n}^{-1}\cdots A_{i_0}^{-1}|E^i(\ii)\cap V_{i-1}\|=\ly{i}.
	\end{equation}
	
\end{lemma}

\begin{proof}
	The first assertion follows from the definitions of Oseledets spaces and the Furstenberg measure. On the other hand,
	$$
	\lim_{n\to\infty}\frac{1}{n}\log\|A_{i_n}^{-1}\cdots A_{i_0}^{-1}|E^i(\ii)\|=\ly{i}\quad \text{and}\quad\lim_{n\to\infty}\frac{1}{n}\log\mathfrak{m}(A_{i_n}^{-1}\cdots A_{i_0}^{-1}|V_j)=\ly{j+1}.
	$$
	Hence,
	\begin{equation*}
	\ly{j+1}\leq\lim_{n\to\infty}\frac{1}{n}\log\mathfrak{m}(A_{i_n}^{-1}\cdots A_{i_0}^{-1}|V_j\cap E^i(\ii))\leq\lim_{n\to\infty}\frac{1}{n}\log\|A_{i_n}^{-1}\cdots A_{i_0}^{-1}|E^i(\ii)\cap V_j\|\leq\ly{i}.
	\end{equation*}
	Thus, if $j+1> i$ then $E^i(\ii)\cap V_j=\{0\}$ almost surely. Moreover, $E^j(\ii)\oplus V_j=\R^d$. If $j+1\leq i$ then, by $E^j(\ii)\subset E^i(\ii)$ and $E^j(\ii)\oplus V_j=\R^d$, we have $E^i(\ii)\supset E^j(\ii)\oplus( E^i(\ii)\cap V_j)$, which implies that $\dim E^i(\ii)\cap V_j\leq i-j$.
\end{proof}

Let us define families of subspaces in $\R^{d+1}$ such that
$$
F_{d}(\theta)=\langle(0,\ldots,0,1)\rangle, \quad F_j(\theta)=F_d\times V_j, \quad \text{and} \quad F_0(\theta)=\R^{d+1},
$$
where $0$ is repeated $d$ times and $\theta=(V_{d-1},\dots,V_1)$. Note that $F_d$ and $F_0$ are independent of $\theta$.
For each $i\in\{0,\dots,d\}$, let $\xi^i$ be the partition of $\widehat{\Lambda}\times\mathbb{F}$ for which
\begin{equation*}
	\xi^i(\widehat{\pi}(\ii),\theta)=\left\{(\widehat{\pi}(\jj),\tau)\in\Omega:\tau=\theta\text{ and }\widehat{\pi}(\jj)-\widehat{\pi}(\ii)\in F_i(\theta)\right\}.
\end{equation*}
Thus, $\xi^0<\xi^1<\cdots<\xi^d$. Moreover, let $\mathcal{P}$ be the partition with respect to the cylinder sets, that is,
$$
\mathcal{P}(\widehat{\pi}(\ii),\theta)=\left\{(\widehat{\pi}(\jj),\tau)\in\Omega:i_0=j_0\right\}.
$$
Let $\mathcal{P}_0^{n-1}=\mathcal{P}\vee\cdots\vee G^{n-1}\mathcal{P}$ be the common refinement of $\mathcal{P}$ at the level $n$, that is,
$$
\mathcal{P}_0^{n-1}(\widehat{\pi}(\ii),\theta)=\left\{(\widehat{\pi}(\jj),\tau)\in\Omega:i_0=j_0,\dots,i_{n-1}=j_{n-1}\right\}.
$$

Observe that, by the uniqueness of the conditional measure, we have
\begin{equation}\label{eq:lambdatomu}
\lambda_{(\widehat{\pi}(\ii),\theta)}^{\xi^i}=\widehat{\mu}_{\widehat{\pi}(\ii)}^{\eta_{\theta}^i}\times\delta_{\theta} \quad \text{and} \quad \lambda_{(\widehat{\pi}(\ii),\theta)}^{\mathcal{P}}=\frac{\left.\widehat{\mu}\right|_{\widehat{f}_{i_0}(\widehat{\Lambda})}}{\widehat{\mu}(\widehat{f}_{i_0}(\widehat{\Lambda}))}\times\mu_F
\end{equation}
for $\lambda$-almost every $(\widehat{\pi}(\ii),\theta)$,
where $\delta_{\theta}$ denotes the Dirac-measure centered at $\theta$ and $\eta_{\theta}^i$ is the partition of $\widehat{\Lambda}$ such that
$$
\eta_{\theta}^i(\widehat{\pi}(\ii))=\left\{\widehat{\pi}(\jj):\widehat{\pi}(\ii)-\widehat{\pi}(\jj)\in F_i(\theta)\right\}.
$$
The transversal ball centred at $(\widehat{\pi}(\ii),\theta)$ with radius $\delta$ is denoted by $$B^T_i((\widehat{\pi}(\ii),\theta),\delta)=B^T_{F_i(\theta)^{\bot}}(\widehat{\pi}(\ii),\delta).$$
Then, by \eqref{eq:Simmons},
\begin{equation}\label{econddef2}
	\widehat{\mu}^{\eta_{\theta}^1}_{\widehat{\pi}(\ii)}=\lim_{\delta\to0+}\frac{\left.\widehat{\mu}\right|_{B_1^T((\widehat{\pi}(\ii),\theta),\delta)}}{\widehat{\mu}(B_1^T((\widehat{\pi}(\ii),\theta),\delta))}\quad \text{and} \quad \widehat{\mu}_{\widehat{\pi}(\ii)}^{\eta^i_{\theta}}=\lim_{\delta\to0+}\frac{\left.\widehat{\mu}^{\eta^{i-1}_{\theta}}_{\widehat{\pi}(\ii)}\right|_{B^T_{i}((\widehat{\pi}(\ii),\theta),\delta)}}{\widehat{\mu}^{\eta^{i-1}_{\theta}}_{\widehat{\pi}(\ii)}(B^T_i((\widehat{\pi}(\ii),\theta),\delta))}
\end{equation}
for $i\geq2$ and $\nu$-almost all $\ii\in\Sigma$.

\begin{lemma}\label{lem:invariance}
	We have $(G^k)_*\bigl(\lambda_{x}^{\xi^i\vee\mathcal{P}_0^{k-1}}\bigr)=\lambda_{G^k(x)}^{\xi^i}$ for all $k\geq1$ and $i \in \{ 0,\dots,d \}$, and for $\lambda$-almost every $x$.
\end{lemma}

\begin{proof}
	Observe that, by \eqref{econdcond} and \eqref{eq:lambdatomu},
	\begin{equation}\label{eq:tech1}
	\lambda_x^{\xi^i\vee\mathcal{P}_0^{k-1}}=\biggl(\frac{\left.\widehat{\mu}\right|_{\widehat{f}_{\iiv}(\widehat{\Lambda})}}{\widehat{\mu}(\widehat{f}_{\iiv}(\widehat{\Lambda}))}\biggr)_{\widehat{\pi}(\ii)}^{\eta_{\theta}^i}\times\delta_{\theta}
	\end{equation}
	for $\lambda$-almost every $x=(\widehat{\pi}(\ii),\theta)$,
	where $\ii|_{k-1}=\iiv$. By the definition of the self-affine measure
	$$
	\widehat{\mu}=\frac{(\widehat{f}_{\iiv}^{-1})_*\left.\widehat{\mu}\right|_{\widehat{f}_{\iiv}(\widehat{\Lambda})}}{\widehat{\mu}(\widehat{f}_{\iiv}(\widehat{\Lambda}))}.
	$$
	Observe that $\widehat{f}_{\iiv}\colon \widehat{\Lambda}\to \widehat{f}_{\iiv}(\widehat{\Lambda})$ is an affine bijection and therefore, by Lemma~\ref{lem:condtech} and \eqref{eq:tech1},
	\begin{align*}
	\biggl(\frac{(\widehat{f}_{\iiv}^{-1})_*\left.\widehat{\mu}\right|_{\widehat{f}_{\iiv}(\widehat{\Lambda})}}{\widehat{\mu}(\widehat{f}_{\iiv}(\widehat{\Lambda}))}\biggr)_{\widehat{\pi}(\sigma^k\ii)}^{\widehat{f}_{\iiv}\eta_{\theta}^i}\times\delta_{A_{\iiv}^{-1}\theta}&=(\widehat{f}_{\iiv}^{-1})_*\biggl(\frac{\left.\widehat{\mu}\right|_{\widehat{f}_{\iiv}(\widehat{\Lambda})}}{\widehat{\mu}(\widehat{f}_{\iiv}(\widehat{\Lambda}))}\biggr)_{\widehat{f}_{\iiv}(\widehat{\pi}(\sigma^k\ii))}^{\eta_{\theta}^i}\times(A_{\iiv}^{-1})_*\delta_{\theta}\\&=(G^k)_*\lambda^{\xi^i\vee\mathcal{P}_0^{k-1}}_{(\widehat{\pi}(\ii),\theta)}.
	\end{align*}
	But on the other hand, $(\widehat{f}_{\iiv}\eta_{\theta}^i)(\widehat{\pi}(\sigma^k\ii))=\widehat{f}_{\iiv}^{-1}(\eta_{\theta}^i(\widehat{\pi}(\ii)))=\eta_{A_{\iiv}^{-1}\theta}^i(\widehat{\pi}(\sigma^k\ii))$, and hence, by \eqref{eq:lambdatomu},
	$$
	\biggl(\frac{(\widehat{f}_{\iiv}^{-1})_*\left.\widehat{\mu}\right|_{\widehat{f}_{\iiv}(\widehat{\Lambda})}}{\widehat{\mu}(\widehat{f}_{\iiv}(\widehat{\Lambda}))}\biggr)_{\widehat{\pi}(\sigma^k\ii)}^{\widehat{f}_{\iiv}\eta_{\theta}^i}\times\delta_{A_{\iiv}^{-1}\theta}=\left(\widehat{\mu}\right)^{\eta_{A_{\iiv}^{-1}\theta}^i}_{\widehat{\pi}(\sigma^k\ii)}\times\delta_{A_{\iiv}^{-1}\theta}=\lambda_{(\widehat{\pi}(\sigma^k\ii),A_{\iiv}^{-1}\theta)}^{\xi^i}.
	$$
	The proof is finished.
\end{proof}

\section{Proof of Ledrappier-Young formula with simple Lyapunov spectrum}\label{sec:proofLY2}

Throughout this section we always assume that $\nu$ is a Bernoulli measure with simple Lyapunov spectrum.
We denote the conditional entropy of $\mathcal{P}$ with respect to $\xi^n$  by $H^n=H(\mathcal{P}|\xi^n)$. We call the measure $(\lambda_{x}^{\xi^i})^T_{F_{i+1}(\theta)^{\bot}}$ a \emph{transversal measure} of $\lambda_{x}^{\xi^i}$.

\begin{lemma}\label{ltoLY1}
	There is a constant $c>0$ such that
	\[
	\xi^d(x)\cap B_{c^{-1}\rho^n}(x)\subseteq(\mathcal{P}_{0}^{n-1}\vee\xi^{d})(x)\subseteq B_{c\rho^n}(x)\cap\xi^d(x).
	\]
	for all $n\geq1$ and $x\in\Omega$.
\end{lemma}

\begin{proof}
	Let us fix a $n\geq1$ and $x=(\widehat{\pi}(\ii),\theta)\in\Omega$. By the definition of $\xi^d$ and $F^d(\theta)$, we have
	$$\mathrm{diam}((\mathcal{P}_{0}^{n-1}\vee\xi^{d})(x))\leq \mathrm{diam}(\widehat{\Lambda})\rho^n.$$
	On the other hand, since the IFS $\widehat{\Phi}$ satisfies the SSC,  $\kappa=\min_{i\neq j}\mathrm{dist}(\widehat{f}_i(\widehat{\Lambda}),\widehat{f}_j(\widehat{\Lambda}))/2>0$. Thus for every $G^{n}(x)\in\Omega$, if $\widehat{\pi}(\sigma^n\ii)\in \widehat{f}_i(\widehat{\Lambda})$, then $\mathrm{dist}(\widehat{\pi}(\sigma^n\ii),\widehat{f}_j(\widehat{\Lambda}))>\kappa$ for every $j\neq i$. So
	\begin{align*}
	\xi^d(x)\cap G^{-n}(B_{\kappa}(G^{n}(x)))\cap\left(\mathcal{P}_{0}^{n-1}\right)(x)&\supseteq G^{-n}\left(B_{\kappa}(G^{n}(x))\cap \xi^d(G^{n}(x))\right)\cap\left(\mathcal{P}_{0}^{n-1}\right)(x)\\&\supseteq B_{\kappa\rho^n}(x)\cap\xi^d(x).	
	\end{align*}
	The statement follows by choosing $c=\max\{\mathrm{diam}(\widehat{\Lambda}),\kappa^{-1}\}$.
\end{proof}

\begin{prop}\label{prop:1}
	The measure $\lambda_{x}^{\xi^d}$ is exact dimensional and
	$$
	\dim\lambda_{x}^{\xi^d}=\frac{H^d}{-\log\rho}
	$$
	for $\lambda$-almost every $x$.
\end{prop}

\begin{proof}
	By Lemma~\ref{ltoLY1}, to prove the statement of the proposition it is enough to show that
	\begin{equation*}
	\lim_{n\rightarrow\infty}\frac{1}{n}\log\lambda_{x}^{\xi^{d}}((\mathcal{P}_{0}^{n-1}\vee\xi^{d})(x))=H^d
	\end{equation*}
	for $\lambda$-almost every $x$.
	Observe that we have
	\begin{align*}	\log\lambda_{x}^{\xi^{d}}((\mathcal{P}_{0}^{n-1}\vee\xi^{d})(x))&=\log\lambda_{x}^{\xi^{d}}\bigl(\mathcal{P}(x)\cap\cdots\cap G^{-(n-1)}(\mathcal{P}(G^{n-1}(x)))\bigr)
	\\&=\log\lambda_{x}^{\xi^{d}}(\mathcal{P}(x))\prod_{k=1}^{n-1}\frac{\lambda_{x}^{\xi^{d}}\left(\mathcal{P}(x)\cap\cdots\cap G^{-k}(\mathcal{P}(G^{k}(x)))\right)}{\lambda_{x}^{\xi^{d}}\left(\mathcal{P}(x)\cap\cdots\cap G^{-(k-1)}(\mathcal{P}(G^{k-1}(x)))\right)}.
	\end{align*}
	By using \eqref{econdcond} and Lemma~\ref{lem:invariance}, we get
	\begin{equation*}
	\frac{\lambda_{x}^{\xi^{d}}\left(\mathcal{P}(x)\cap\cdots\cap G^{-k}(\mathcal{P}(G^{k}(x)))\right)}{\lambda_{x}^{\xi^{d}}\left(\mathcal{P}(x)\cap\cdots\cap G^{-(k-1)}(\mathcal{P}(G^{k-1}(x)))\right)}=\lambda_{x}^{\xi^{d}\vee\mathcal{P}_0^{k-1}}\bigl(G^{-k}(\mathcal{P}(G^{k}(x)))\bigr)=
	\lambda_{G^{k}(x)}^{\xi^{d}}\bigl(\mathcal{P}(G^{k}(x)))\bigr).
	\end{equation*}
	Hence, by Birkhoff's Ergodic Theorem,
	\begin{equation*}
	\frac{1}{n}\log\lambda_{x}^{\xi^{d}}((\mathcal{P}_{0}^{n-1}\vee\xi^{d})(x))=\frac{1}{n}\sum_{k=0}^{n-1}\log\cm{G^{k}(x)}{\xi^{d}}\bigl(\mathcal{P}(G^{k}(x)))\bigr)\to\int\log\lambda_{y}^{\xi^{d}}(\mathcal{P}(y))\dd\lambda(y),
	\end{equation*}
	for $\lambda$-almost every $x$.
\end{proof}

\begin{prop}\label{pLY2}
	For $\lambda$-almost every $x=(\widehat{\pi}(\ii),\theta)$, the measure $(\lambda_{x}^{\xi^i})^T_{F_{i+1}(\theta)^{\bot}}$ is exact dimensional and
	$$
	\dim(\lambda_{x}^{\xi^i})^T_{F_{i+1}(\theta)^{\bot}}=\frac{H^i-H^{i+1}}{\ly{i+1}}.
	$$
\end{prop}

We note that the measure $(\lambda_{x}^{\xi^i})^T_{(F_{i+1}(\theta))^{\bot}}$ is the orthogonal projection of the measure $\lambda_{x}^{\xi^i}$ onto the orthogonal complement of $F_{i+1}(\theta)$ and $\dim(\left(F_{i+1}(\theta)\right)^{\bot}\cap F_{i}(\theta))=1$.
Let us introduce modified transversal balls $B^t_i(x,\delta)$ for $\lambda$-almost every $x$ by setting
\begin{align*}
	B^t_i(x,\delta)&=\left\{y\in\xi^{i-1}(x):\mathrm{dist}(E^{i}(x)\cap\xi^i(x), E^{i}(x)\cap\xi^i(y))<\delta\right\}, \\
	B^t_d(x,\delta)&=B^T_d(x,\delta)\cap\xi^{d-1}(x),
\end{align*}
where $E^i(x)=E^i(\ii)$ is defined in Theorem~\ref{thm:Oseledets}. By the definition,
$$
B^t_i(x,\delta)=B^T_i(x,\delta\cdot\cos\sphericalangle(E^i(x)\cap V_{i-1},V_i^{\bot}\cap V_{i-1}))
$$
for $\lambda$-almost every $x\in\Omega$ and for every $\delta>0$. Let us define functions
\begin{equation*}
	w^{i}(x)=\cm{x}{\xi^{i}}(\mathcal{P}(x))
\end{equation*}
and
\begin{equation}\label{eq:wdelta}
w_{\delta}^{i}(x)=\frac{\cm{x}{\xi^{i-1}}(B^t_{i}(x,\delta)\cap\mathcal{P}(x))}{\cm{x}{\xi^{i-1}}(B^t_{i}(x,\delta))}
\end{equation}
for all $i \in \{1,\dots,d\}$. By \eqref{econddef2}, $w_{\delta}^{i}\rightarrow w^{i}$ as $\delta\rightarrow0+$ for $\lambda$-almost everywhere and, since $w_{\delta}^{i}$ is uniformly bounded, $w_{\delta}^i\rightarrow w^i$ in $L^1(\lambda)$ as $\delta\rightarrow0+$ for all $i$.

\begin{lemma}\label{ltoLY2b}
	The function $\sup_{\delta>0}\{-\log w_{\delta}^{i}\}$ is in $L^1(\lambda)$ for every $i$.
\end{lemma}

\begin{proof}
  The proof is a slight modification of the proof of \cite[Lemma~3.6]{Ba}.
\end{proof}

\begin{theorem}[Maker \cite{M}]\label{lmaker}
	Let $T\colon X\to X$ be an endomorphism on a compact set $X\subset\R^d$ and let $m$ be a $T$-invariant ergodic measure. Moreover, let $h_{p,k}\colon X\to\R$ be a family of functions for which $\sup_{p,l}h_{p,l}\in L^1(m)$ and $\lim_{p-l\rightarrow\infty}h_{p,l}=h$ in $L^1(m)$ and for $m$-almost everywhere, where $h\in L^1(m)$. Then
	\[
	\lim_{p\rightarrow\infty}\frac{1}{p}\sum_{l=0}^{p-1}h_{p,l}(T^lx)=\int h(x)\dd m(x)
	\]
	for $m$-almost every $x\in X$.
\end{theorem}

\begin{lemma}\label{lem:econt2}
	For $\lambda$-almost every $x=(\widehat{\pi}(i_0i_1\dots),\theta)$ we have
	\begin{equation*}
		G^{-1}\left(B^t_{i}\left(G(x),\delta\right)\right)\cap\mathcal{P}(x)=B^t_{i}(x,\delta\|A_{i_0}|E^i(G(x))\cap A_{i_0}^{-1}V_{i-1}\|)\cap\mathcal{P}(x).
	\end{equation*}
\end{lemma}

\begin{proof}
	By the definition of $B^t_{i}(x,\delta)$,
	$$
	G^{-1}\left(B^t_{i}\left(G(x),\delta\right)\right)=\bigcup_{k=1}^NG_k\left(\{y\in\xi^{i-1}(G(x)):\mathrm{dist}(E^i(G(x))\cap\xi^{i}(G(x)),E^i(G(x))\cap\xi^{i}(y))<\delta\right\}),
	$$
	where $G_k$ is the local inverse of $G$, i.e.\ $G_k(x,\theta)=(\widehat{f}_k(x),A_k\theta)$.

	On the other hand, $G_k(A)\subseteq\mathcal{P}(G_k(x))$ for any $x\in\Omega$ and $A$ subset of $\Omega$ and therefore
	$$
	G^{-1}\left(B^t_{i}\left(G(x),\delta\right)\right)=\bigcup_{k=1}^NB^t_{i}(G_k(G(x)),\delta\|A_{k}|E^i(G(x))\cap A_{i_0}^{-1}V_{i-1}\|)\cap\mathcal{P}(G_k(G(x))).
	$$
	Since $\mathcal{P}$ is a partition, the statement follows.
\end{proof}

\begin{lemma}\label{lem:angle}
We have
$$
\lim_{n\to\infty}\frac{1}{n}\log \cos\sphericalangle(E^i(G^n(x))\cap A_{i_n}^{-1}\cdots A_{i_0}^{-1}V_{i-1},(A_{i_n}^{-1}\cdots A_{i_0}^{-1}V_i)^{\bot}\cap A_{i_n}^{-1}\cdots A_{i_0}^{-1}V_{i-1})=0
$$
	for $\nu\times\mu_F$-almost every $(\ii,\theta)$.
\end{lemma}

\begin{proof}
	Since $\dim E^i(\sigma^n\ii)\cap A_{i_n}^{-1}\cdots A_{i_0}^{-1}V_{i-1})=\dim A_{i_0}^{-1}V_{i-1},(A_{i_n}^{-1}\cdots A_{i_0}^{-1}V_i)^{\bot}\cap A_{i_n}^{-1}\cdots A_{i_0}^{-1}V_{i-1}=1$, we get, by the definition angles between subspaces, that
	\begin{align*}
		\cos\sphericalangle(&E^i(G^n(x))\cap A_{i_n}^{-1}\cdots A_{i_0}^{-1}V_{i-1},(A_{i_n}^{-1}\cdots A_{i_0}^{-1}V_i)^{\bot}\cap A_{i_n}^{-1}\cdots A_{i_0}^{-1}V_{i-1})\\
		&=\sin\sphericalangle(E^i(G^n(x))\cap A_{i_n}^{-1}\cdots A_{i_0}^{-1}V_{i-1},A_{i_n}^{-1}\cdots A_{i_0}^{-1}V_i)=
		\frac{\|v\wedge u_1\wedge\cdots\wedge u_{d-i}\|}{\|v\|\|u_1\wedge\cdots\wedge u_{d-i}\|},
	\end{align*}
	where $\mathrm{span}\{u_1,\dots,u_{d-i}\}=A_{i_n}^{-1}\cdots A_{i_0}^{-1}V_i$ and $0\neq v\in E^i(G^n(x))\cap A_{i_n}^{-1}\cdots A_{i_0}^{-1}V_{i-1}$. Furthermore, let $\{v_1,\dots,v_{d-i+1}\}$ be vectors from $V_{i-1}$ such that
	$$
	v_1\in E^i(\ii)\cap V_{i-1}\quad \text{ and }\quad \mathrm{span}\{v_2,\dots,v_{d-i+1}\}=V_i.
	$$
	Thus,
	\begin{multline*}
		\frac{\|(A_{i_n}^{-1}\cdots A_{i_0}^{-1}v_1)\wedge \cdots\wedge (A_{i_n}^{-1}\cdots A_{i_0}^{-1}v_{d-i+1})\|}{\|(A_{i_n}^{-1}\cdots A_{i_0}^{-1}v_1\|\|(A_{i_n}^{-1}\cdots A_{i_0}^{-1}v_2)\wedge \cdots\wedge (A_{i_n}^{-1}\cdots A_{i_0}^{-1}v_{d-i+1})\|}\\
		=\frac{\|\bigwedge^{d-i+1}A_{i_n}^{-1}\cdots A_{i_0}^{-1}|\bigwedge^{d-i+1}V_{i-1}\|}{\|(A_{i_n}^{-1}\cdots A_{i_0}^{-1}|E^i(\ii)\cap V_{i-1}\|\|\bigwedge^{d-i} A_{i_n}^{-1}\cdots A_{i_0}^{-1}|\bigwedge^{d-i}V_i\|}\cdot\frac{\|v_1\wedge\cdots\wedge v_{d-i+1}\|}{\|v_1\|\|v_2\wedge\cdots\wedge v_{d-i+1}\|}.
	\end{multline*}
	By the definition of the Furstenberg measure and Oseledec's Theorem,
	 \begin{equation}\label{eq:what}
	 \lambda(\{(\widehat{\pi}(\ii),\theta):\cos\sphericalangle(E^i(x)\cap V_{i-1},V_i^{\bot}\cap V_{i-1})=0\})=0.
	 \end{equation}
	Thus the assertion follows from Theorem~\ref{thm:furst} and Lemma~\ref{lem:angle&dim}.
\end{proof}

To simplify notation, we denote the subspace $E^i(x)\cap V_{i-1}$ by $K^i(x)$ for all $i \in \{2,\dots,d-1\}$. We also write $K^d(x)=F_d(\theta)$ and $K^1(x)=E^1(x)$.

\begin{proof}[Proof of Proposition~\ref{pLY2}]
	By \eqref{eq:what}, we may assume that $\cos\sphericalangle(E^i(x)\cap V_{i-1},V_i^{\bot}\cap V_{i-1})\neq0$.  Then, by the definition of the transversal measure and the transversal ball, it is enough to show that
	\begin{equation*}
	\lim_{\delta\rightarrow0+}\frac{\log\cm{x}{\xi^i}(B^t_{i+1}(x,\delta))}{\log\delta}=\frac{H^i-H^{i+1}}{\ly{i+1}}
	\end{equation*}
	for $\lambda$-almost every $x$. Because of the exponential shrinking rate, this is equivalent to
	\begin{equation}\label{elocdimLY2}
	\lim_{n\rightarrow\infty}\frac{\log\cm{x}{\xi^i}\left(B^t_{i+1}(x,\|A_{i_{0}}\cdots A_{i_{n-1}}|K^{i+1}(G^{n}(x))\|)\right)}{\log\|A_{i_{0}}\cdots A_{i_{n-1}}|K^{i+1}(G^n(x))\|}=\frac{H^i-H^{i+1}}{\ly{i+1}}
	\end{equation}
	for $\lambda$-almost every $x$. We write the measure of the ball as
	\begin{equation}\label{etech3}
	\begin{split}
	\cm{x}{\xi^i}&\left(B^t_{i+1}(x,\|A_{i_{0}}\cdots A_{i_{n-1}}|K^{i+1}(G^{n}(x))\|)\right)\\
	&\quad=\cm{G^{n}(x)}{\xi^{i}}\left(B^t_{i+1}(G^{n}(x),1)\right)\cdot\prod_{l=1}^n\dfrac{\cm{G^{l-1}(x)}{\xi^{i}}\left(B^t_{i+1}(G^{l-1}(x),\|A_{i_{l-1}}\cdots A_{i_{n-1}}|K^{i+1}(G^{n}(x))\|)\right)}{\cm{G^{l}(x)}{\xi^{i}}\left(B^t_{i+1}(G^{l}(x),\|A_{i_{l}}\cdots A_{i_{n-1}}|K^{i+1}(G^{n}(x))\|)\right)}.
	\end{split}
	\end{equation}
	In the above calculation we interpret $\|A_{i_{l}}\cdots A_{i_{n-1}}|K^{i+1}(G^{n}(x))\|=1$ when $l=n$.
	
Observe that, by \eqref{econdcond} and Lemma~\ref{lem:invariance}, we have
\begin{equation}\label{econt1}
\frac{\cm{x}{\xi^{i}}(G^{-1}(B^t_{i+1}(G(x),\delta))\cap\mathcal{P}(x))}{\cm{x}{\xi^{i}}(\mathcal{P}(x))}=\cm{x}{\xi^i\vee\mathcal{P}}(G^{-1}(B^t_{i+1}(G(x),\delta)))=\cm{G(x)}{\xi^{i}}(B^t_{i+1}(G(x),\delta))).
\end{equation}
	Applying \eqref{econt1} and Lemma~\ref{lem:econt2}, we get
	\begin{align}
	&\cm{G^{l}(x)}{\xi^{i}}\left(B^t_{i+1}(G^{l}(x),\|A_{i_{l}}\cdots A_{i_{n-1}}|K^{i+1}(G^{n}(x))\|)\right)\notag\\
	&\;\quad=\cm{G^{l-1}(x)}{\xi^{i}\vee\mathcal{P}}\left(G^{-1}(B^t_{i+1}(G^{l}(x),\|A_{i_{l}}\cdots A_{i_{n-1}}|K^{i+1}(G^{n}(x))\|))\right)\notag\\
	&\;\quad=\dfrac{\cm{G^{l-1}(x)}{\xi^{i}}\left(G^{-1}(B^t_{i+1}(G^{l}(x),\|A_{i_{l}}\cdots A_{i_{n-1}}|K^{i+1}(G^{n}(x))\|)\cap\mathcal{P}(G^{l-1}(x)))\right)}{\cm{G^{l-1}(x)}{\xi^{i}}(\mathcal{P}(G^{l-1}(x)))}\label{etech2}\\
	&\;\quad=\dfrac{\cm{G^{l-1}(x)}{\xi^{i}}\left(B^t_{i+1}(G^{l-1}(x),\|A_{i_{l-1}}|K^{i+1}(G^l(x))\|\|A_{i_{l}}\cdots A_{i_{n-1}}|K^{i+1}(G^{n}(x))\|)\cap\mathcal{P}(G^{l-1}(x)))\right)}{\cm{G^{l-1}(x)}{\xi^{i}}(\mathcal{P}(G^{l-1}(x)))}\notag
	\end{align}
	for every $l\in\{1,\dots,n\}$.
	Since $K^{i+1}$'s are one dimensional invariant subspaces, $$\|A_{i_{l-1}}|K^{i+1}(G^l(x))\|\|A_{i_{l}}\cdots A_{i_{n-1}}|K^{i+1}(G^{n}(x))\|=\|A_{i_{l-1}}A_{i_{l}}\cdots A_{i_{n-1}}|K^{i+1}(G^{n}(x))\|.$$ Hence, applying \eqref{etech2} in the denominator of \eqref{etech3}, we get	
	\begin{multline*}
	\frac{1}{n}\log\cm{x}{\xi^i}\left(B^t_{i+1}(x,\|A_{i_{0}}\cdots A_{i_{n-1}}|K^{i+1}(G^{n}(x))\|)\right)=
	\frac{1}{n}\log\cm{G^{n}(x)}{\xi^{i}}\left(B^t_{i+1}(G^{n}(x),1)\right)\\
	-\frac{1}{n}\sum_{l=1}^n\log w^{i+1}_{\|A_{i_{l-1}}\cdots A_{i_{n-1}}|K^{i+1}(G^{n}(x))\|}(G^{l-1}(x))+\frac{1}{n}\sum_{l=1}^n\log\cm{G^{l-1}(x)}{\xi^{i}}(\mathcal{P}(G^{l-1}(x))),
	\end{multline*}
	where $w^{i+1}_{\delta}$ was defined in \eqref{eq:wdelta}. Let us define a function $$h_{n,k}^i(x)=\log w^{i+1}_{\|A_{i_{0}}\cdots A_{i_{n-k}}|K^{i+1}(G^{n-k+1}(x))\|}(x).$$ Then
	\begin{equation}\label{eq:entropy}
	\begin{split}
	\frac{1}{n}\log\cm{x}{\xi^i}&\left(B^t_{i+1}(x,\|A_{i_{0}}\cdots A_{i_{n}}|K^i(G^{n}(x))\|)\right)=
	\frac{1}{n}\log\cm{G^{n}(x)}{\xi^{i}}\left(B^t_{i+1}(G^{n}(x),1)\right)\\
	&\quad\quad\quad\quad\quad\quad\quad-\frac{1}{n}\sum_{l=1}^nh_{n,l}^i(G^{l-1}(x))+\frac{1}{n}\sum_{l=1}^n\log\cm{G^{l-1}(x)}{\xi^{i}}(\mathcal{P}(G^{l-1}(x))).
	\end{split}
	\end{equation}
	Since $\|A_{i_{0}}\cdots A_{i_{n-l}}|K^i(G^{n-l+1}(x))\|\rightarrow0$ uniformly on $\Omega$ as $n-l\rightarrow\infty$, we get that $\lim_{n-l\rightarrow\infty}h_{n,l}^i=\log w^{i+1}$ in $L^1(\lambda)$ and for $\lambda$-almost everywhere. Furthermore, by Lemma~\ref{ltoLY2b}, we can apply Maker's Ergodic Theorem (Theorem~\ref{lmaker}) and hence
	\begin{equation}\label{eq:tofin1}
	\lim_{n\rightarrow\infty}\frac{1}{n}\sum_{l=1}^nh_{n,l}^i(G^{l-1}(x))=H^{i+1}
	\end{equation}
	for $\lambda$-almost every $x$.
	On the other hand, by Birkhoff's Ergodic Theorem
	\begin{equation}\label{eq:tofin2}
		\lim_{n\rightarrow\infty}\frac{1}{n}\sum_{l=1}^n\log\cm{G^{l-1}(x)}{\xi^{i}}(\mathcal{P}(G^{l-1}(x)))=-H^{i}
	\end{equation}
	for $\lambda$-almost every $x$. Finally, we note that
	\begin{align*}
	\lim_{n\rightarrow\infty}\frac{1}{n}&\log\cm{G^{n}(x)}{\xi^{i}}\left(B^t_{i+1}(G^{n}(x),1)\right)\\
	&=\lim_{n\rightarrow\infty}\frac{1}{n}\log\cm{G^{n}(x)}{\xi^{i}}\left(B^T_{i+1}(G^{n}(x),\sin\sphericalangle(K^{i+1}(G^n(x)),F_{i+1}(G^n(x)))\right).
	\end{align*}
	If $i=d-1$, then we clearly have $\cos\sphericalangle(K^i(G^n(x)),F_{i+1}(G^n(x)))=\sin\sphericalangle(\theta,\langle(0,\ldots,0,1)\rangle)=1$ for every $n\geq1$ and $x\in\Omega$. Therefore
	$$
	\lim_{n\rightarrow\infty}\frac{1}{n}\log\cm{G^{n}(x)}{\xi^{d-1}}\left(B^T_{d}(G^{n}(x),1)\right)=0
	$$
	for $\lambda$-almost every $x\in\Omega$.
	
	Let us denote the subspace $(A_{i_n}^{-1}\cdots A_{i_0}^{-1}V_i)^{\bot}\cap A_{i_n}^{-1}\cdots A_{i_0}^{-1}V_{i-1}$ by $L^i(G^n(x))$.
	On the other hand, if $0\leq i\leq d-2$, then, by Lemma~\ref{lem:angle}, for every $\varepsilon>0$ there exists an $n_0=n_0(x,\varepsilon)>0$ such that
	$$
	e^{-\varepsilon n}\leq\sin\sphericalangle(K^{i+1}(G^n(x)),L^{i+1}(G^n(x)))\leq1
	$$
	for every $n\geq n_0$.
	Hence,
	\begin{align*}
	0&\geq\limsup_{n\rightarrow\infty}\frac{1}{n}\log\cm{G^{n}(x)}{\xi^{i}}\left(B^T_{i+1}(G^{n}(x),\sin\sphericalangle(K^{i+1}(G^n(x)),L^{i+1}(G^n(x)))\right)\\
	&\geq \varepsilon\liminf_{n\rightarrow\infty}\frac{1}{\varepsilon n}\log\cm{G^{n}(x)}{\xi^{i}}\left(B^T_{i+1}(G^{n}(x),e^{-\varepsilon n}\right)\geq\varepsilon\liminf_{n\rightarrow\infty}\frac{1}{\varepsilon n}\log\cm{G^{n}(x)}{\xi^{i}}\bigl(\mathcal{P}_0^{\lceil c\varepsilon n\rceil}(G^{n}(x))\bigr)\\
	&\geq\varepsilon c\log p_{\min},
	\end{align*}
	where $c=\max_{i\in\{1,\dots,N\}}\lceil(-\log\|A_i\|)^{-1}\rceil$ and $p_{\min}=\min_{i\in\{1,\dots,N\}}p_i$. Since the inequality holds for any $\varepsilon>0$ we get
	\begin{equation}\label{eq:tofin3}
	\lim_{n\rightarrow\infty}\frac{1}{n}\log\cm{G^{n}(x)}{\xi^{i}}\left(B^t_{i+1}(G^{n}(x),1)\right)=0
	\end{equation}
	for every $i$ and $\lambda$-almost every $x$.
	Now \eqref{eq:entropy}, together with \eqref{eq:tofin1}, \eqref{eq:tofin2}, and \eqref{eq:tofin3}, implies
	\begin{equation}\label{eentLY2}
	\lim_{n\rightarrow\infty}\frac{1}{n}\log\cm{x}{\xi^i}\left(B^t_{i+1}(x,\|A_{i_{0}}\cdots A_{i_{n}}|K^{i+1}(G^{n}(x))\|)\right)=-H^i+H^{i+1}
	\end{equation}
	for $\lambda$-almost every $x$.
	
	To finish the proof, observe that for $x=(\widehat{\pi}(\ii),\theta)$ we have
	$$
	\|A_{i_0}\cdots A_{i_n}|K^{i+1}(G^n(x))\|=\|A_{i_n}^{-1}\cdots A_{i_0}^{-1}|K^{i+1}(x)\|^{-1}
	$$
	and hence, by Theorem~\ref{thm:Oseledets} and \eqref{eq:Furstenberg},
	$$\lim_{n\to\infty}\frac{1}{n}\log\|A_{i_{0}}\cdots A_{i_{n}}|K^{i+1}(G^n(x)\|=-\ly{i+1}$$
	for $\lambda$-almost every $x$.
	This together with \eqref{eentLY2} implies \eqref{elocdimLY2}.
\end{proof}

\begin{prop}\label{prop:LY3}
	For $\lambda$-almost every $x=(\widehat{\pi}(\ii),\theta)$ the measure $\lambda_{x}^{\xi^i}$ is exact dimensional and
	$$
	\dim\lambda_{x}^{\xi^i}=\frac{H^{d}}{-\log\rho}+\sum_{k=i}^{d-1}\frac{H^k-H^{k+1}}{\ly{k+1}},
	$$
	for all $i\leq d-1$.
\end{prop}

\begin{proof}
		We prove the theorem by induction. First, we show the statement for the case $i=d-1$. Since the proof of the first step and the proof of the inductional step does not differ much, we use $i$ instead of $d-1$. By the uniqueness of conditional measures and \eqref{econdcond},
		$$
		\cm{x}{\xi^{i}}=\int\cm{y}{\xi^{i+1}}\dd\cm{x}{\xi^{i}}(y)
		$$
		for $\lambda$-almost every $x$.	Thus, choosing $V=F_{i+1}(\theta)^{\bot}$ in Lemma~\ref{llbmuubtmu} and recalling Proposition~\ref{prop:1} and Proposition~\ref{pLY2}, we get
		$$
		\ldimloc(\cm{x}{\xi^{i}},x)\geq\ldimloc(\cm{x}{\xi^{i+1}},x)+\dim(\lambda_{x}^{\xi^i})^T_{F_{i+1}(\theta)^{\bot}}=\frac{H^{d}}{-\log\rho}+\sum_{k=i}^{d-1}\frac{H^k-H^{k+1}}{\ly{k+1}}
		$$
		for $\lambda$-almost every $x$.
				
		Let us next show the upper bound. For simplicity, let $D_{i}=\frac{H^{d}}{-\log\rho}+\sum_{k=i}^{d-1}\frac{H^k-H^{k+1}}{\ly{k+1}}$. By Egorov's theorem there exists a set $J_1$ such that $\lambda(J_1)>1-\varepsilon$ and
		\begin{align}
		\cm{x}{\xi^{i+1}}(B(x,e^{-m(\ly{i+1}-2\varepsilon)}))&\geq e^{-m(\ly{i+1}-2\varepsilon)(D_{i+1}+\varepsilon)}\label{ely31} \\
		\cm{x}{\xi^{i+1}}(\mathcal{P}_0^{m-1}(x))&\leq e^{-m (H^{i+1}-\varepsilon)}\label{ely32}\\
		\cm{x}{\xi^{i}}(\mathcal{P}_0^{m-1}(x))&\geq e^{-m (H^i+\varepsilon)}\label{ely33}\\
		\mathcal{P}_0^{m-1}(x)\cap\xi^i(x)&\subseteq B(x,e^{-m(\ly{i+1}-2\varepsilon)})\label{ely34}
		\end{align}
		for every $x\in J_1$ and every $m$ sufficiently large. Let us denote $J_1 \cap B(x,e^{-m(\ly{i+1}-2\varepsilon)})$ by $L_m(x)$. By Besicovitch's Density Theorem (\cite[Corollary~2.14]{Ma}) and Egorov's Theorem, there exists a set $J_2\subset J_1$ such that $\lambda(J_2)>1-2\varepsilon$ and
		$$
		\frac{\cm{x}{\xi^{i+1}}(L_m(x))}{\cm{x}{\xi^{i+1}}(B(x,e^{-m(\ly{i+1}-2\varepsilon)}))}\geq\frac{1}{2}.
		$$
		for every $x\in J_2$. Thus, by \eqref{ely31} and \eqref{ely32},
		\begin{equation}\label{eforsharp}
                \begin{split}
                \frac{1}{2}e^{-m(\ly{i+1}-2\varepsilon)(D_{i+1}+\varepsilon)}&\leq\cm{x}{\xi^{i+1}}(L_m(x))\\
                &\leq \#\left\{P\in\mathcal{P}_0^{m-1}:P\cap L_m(x)\neq\emptyset\right\}\max_{\substack{P\in\mathcal{P}_0^{m-1} \\ P\cap L_m(x)\neq\emptyset}}\cm{x}{\xi^{i+1}}(P)\\ &\leq\#\left\{P\in\mathcal{P}_0^{m-1}:P\cap L_m(x)\neq\emptyset\right\}e^{-m (H^{i+1}-\varepsilon)}
		\end{split}
		\end{equation}
		for every $x\in J_2$.
		By \eqref{ely33}, \eqref{ely34}, and \eqref{eforsharp}, we have
		\begin{align*}
		\cm{x}{\xi^{i}}(B(x,2e^{-m(\ly{i+1}-2\varepsilon)})) &\geq\#\left\{P\in\mathcal{P}_0^{m-1}:P\cap L_m(x)\neq\emptyset\right\}\min_{\substack{P\in\mathcal{P}_0^{m-1} \\ P\cap L_m(x)\neq\emptyset}}\cm{x}{\xi^{i}}(P)\\
		&\geq\frac{1}{2}e^{-m(\ly{i+1}-2\varepsilon)(D_{i+1}+\varepsilon)}e^{m (H^{i+1}-\varepsilon)}e^{-m (H^i+\varepsilon)}
		\end{align*}
		Hence, for every $x\in J_2$,
		$$
		\overline{\dim}(\cm{x}{\xi^{i}},x)\leq D_{i+1}+\frac{H^i-H^{i+1}}{\ly{i+1}}+o(\varepsilon).
		$$
		Since $\varepsilon>0$ was arbitrary, the statement is proven.
		
		Thus, we have shown that $\lambda_{x}^{\xi^{d-1}}$ is exact dimensional for $i=d-1$. For the indices $i<d-1$, one can show the claim similarly by repeating the previous argument.
	\end{proof}
	
	\begin{prop}\label{pLY4}
		For $\lambda$-almost every $x=(\widehat{\pi}(\ii),\theta)$ the measure $(\lambda_{x}^{\xi^0})^T_{F_{i}(\theta)^{\bot}}$ is exact dimensional and
		$$
		\dim(\lambda_{x}^{\xi^0})^T_{F_{i}(\theta)^{\bot}}=\sum_{k=0}^{i-1}\frac{H^k-H^{k+1}}{\ly{k+1}}.
		$$
	\end{prop}

\begin{proof}
	We prove the statement by induction. For $i=1$ the statement follows from Proposition~\ref{pLY2}. Let $i>1$ and let us assume that the proposition holds for every $j<i$.
	By the uniqueness of the conditional measures and \eqref{econdcond}, we have
	$$
	\lambda_x^{\xi^0}=\int\lambda^{\xi^i}_{y}\dd\mu_{x}^{\xi^0}(y)
	$$
	for $\lambda$-almost every $x$.
	Applying Lemma~\ref{llbmuubtmu} and Proposition~\ref{prop:LY3}, we get
	\begin{align*}
		\udimloc(\lambda_{x}^{\xi^0},x) &=\frac{H^{d}}{\log\rho}+\sum_{j=0}^{d-1}\frac{H^j-H^{j+1}}{\ly{j+1}}\\
		\ &\geq\udimloc(\lambda_{x}^{\xi^i},x)+\udimloc((\lambda_{x}^{\xi^0})^T_{(F_i(\theta))^{\bot}},x)\\
		\ &=\frac{H^{d}}{\log\rho}+\sum_{j=i}^{d-1}\frac{H^j-H^{j+1}}{\ly{j+1}}+\udimloc((\lambda_{x}^{\xi^0})^T_{(F_i(\theta))^{\bot}},x).
	\end{align*}
	Thus,
	\begin{equation*}
		\sum_{j=0}^{i-1}\frac{H^j-H^{j+1}}{\ly{j+1}}\geq\udimloc((\lambda_{x}^{\xi^0})^T_{(F_i(\theta))^{\bot}},x)
	\end{equation*}
	for $\lambda$-almost all $x$.
	On the other hand,
	$$	(\lambda_{x}^{\xi^0})^T_{(F_i(\theta))^{\bot}}=\int(\lambda_{y}^{\xi^{i-1}})^T_{(F_i(\theta))^{\bot}}\dd\lambda_{x}^{\xi^0}(y)
	$$
	and
	$$	((\lambda_{x}^{\xi^0})^T_{(F_i(\theta))^{\bot}})_{(F_{i-1}(\theta))^{\bot}}^T=(\lambda_{x}^{\xi^0})^T_{(F_{i-1}(\theta))^{\bot}}.
	$$
	Thus, applying Lemma~\ref{llbmuubtmu} for $(\lambda^{\xi^0}_{x})^T_{(F_i(\theta))^{\bot}}$, we get, by Proposition~\ref{pLY2} and the induction assumption, that
	\begin{align*}
		\ldimloc((\lambda_{x}^{\xi^0})^T_{(F_i(\theta))^{\bot}},x)&\geq\ldimloc((\lambda_{x}^{\xi^{i-1}})^T_{(F_i(\theta))^{\bot}},x)+\ldimloc(((\lambda_{x}^{\xi^0})^T_{(F_{i}(\theta))^{\bot}})_{(F_{i-1}(\theta))^{\bot}}^T,x)\\
		\
		&=\ldimloc((\lambda_{x}^{\xi^{i-1}})^T_{(F_i(\theta))^{\bot}},x)+\ldimloc((\lambda_{x}^{\xi^0})^T_{(F_{i-1}(\theta))^{\bot}},x)\\
		\
		&=\frac{H^{i-1}-H^{i}}{\ly{i}}+\sum_{k=0}^{i-2}\frac{H^k-H^{k+1}}{\ly{k+1}}
	\end{align*}
	for $\lambda$-almost every $x$.
\end{proof}

\begin{proof}[Proof of Theorem~\ref{texact2}]
	Observe that $\mu=\widehat{\mu}^T_{F_d(\theta)^{\bot}}$.
	By \eqref{econdcond} and \eqref{eq:lambdatomu}, we have
        $$
        \widehat{\mu}=\lambda_{x}^{\xi^0}=\int\lambda^{\xi^d}_{y}\dd\lambda_{x}^{\xi^0}(y)
	$$
	for $\lambda$-almost every $x$.
	Applying Proposition~\ref{pLY4} in the case $i=d$, we get
	\begin{equation*}
	\dim\mu=\sum_{k=0}^{d-1}\frac{H^k-H^{k+1}}{\ly{k+1}}.
	\end{equation*}
	By simple algebra, we see that
	\begin{align*}
		\frac{H^0-H^d}{\ly{d}}+\sum_{i=1}^{d-1}\left(\frac{\ly{i+1}-\ly{i}}{\ly{d}}\right)\sum_{k=0}^{i-1}\frac{H^k-H^{k+1}}{\ly{k+1}}
		&=\frac{H^0-H^d}{\ly{d}}+\sum_{k=0}^{d-2}\sum_{i=k+1}^{d-1}\left(\frac{\ly{i+1}-\ly{i}}{\ly{d}}\right)\frac{H^k-H^{k+1}}{\ly{k+1}}\\
		&=\frac{H^0-H^d}{\ly{d}}+\sum_{k=0}^{d-2}\left(\frac{\ly{d}-\ly{k+1}}{\ly{d}}\right)\frac{H^k-H^{k+1}}{\ly{k+1}}\\
		&=\frac{H^0-H^d}{\ly{d}}+\sum_{k=0}^{d-2}\frac{H^k-H^{k+1}}{\ly{k+1}}-\frac{H^0-H^{d-1}}{\ly{d}}\\
		&=\sum_{j=0}^{d-1}\frac{H^j-H^{j+1}}{\ly{j+1}}.
	\end{align*}
	By Proposition~\ref{pLY4},  $\mu^T_{V_i^{\bot}}=(\cm{x}{\xi^0})^T_{\left(F_{i}(\theta)\right)^{\bot}}$  is exact dimensional for $\mu_F$-almost every $\theta$ and every $i$. Moreover,
	\begin{align*}
		\dim\mu&=\frac{H^0-H^d}{\ly{d}}+\sum_{i=1}^{d-1}\left(\frac{\ly{i+1}-\ly{i}}{\ly{d}}\right)\sum_{k=0}^{i-1}\frac{H^k-H^{k+1}}{\ly{k+1}}\\
		&=\frac{H^0-H^d}{\ly{d}}+\sum_{i=1}^{d-1}\left(\frac{\ly{i+1}-\ly{i}}{\ly{d}}\right)\dim(\cm{x}{\xi^0})^T_{\left(F_{i}(\theta)\right)^{\bot}}\\
		&=\frac{H^0-H^d}{\ly{d}}+\sum_{i=1}^{d-1}\left(\frac{\ly{i+1}-\ly{i}}{\ly{d}}\right)\dim\mu^T_{V_i^{\bot}}.
	\end{align*}
	Finally, we see that
	$$
	H^0=-\int\log\lambda_{x}^{\xi^0}(\mathcal{P}(x))\dd\lambda(x)=-\int\log\widehat{\mu}(\widehat{f}_{i_0}(\widehat{\Lambda}))\dd\nu(\ii)\dd\mu_F(\theta)=h_{\nu}
	$$
	and, by \eqref{eq:lambdatomu},
	$$
	H^d=-\int\log\lambda_{x}^{\xi^d}(\mathcal{P}(x))\dd\lambda(x)=-\int\log\widehat{\mu}^{\eta_{\theta}^d}_{\widehat{\pi}(\ii)}\left(\mathcal{P}((\widehat{\pi}(\ii),\theta))\right)\dd\nu(\ii)\dd\mu_F(\theta).
	$$
	Since, by \cite[Theorem~2.2]{Si},
	\begin{align*}
	\widehat{\mu}_{\widehat{\pi}(\ii)}^{\eta_{\theta}^d}(\mathcal{P}(\widehat{\pi}(\ii),\theta))&=\lim_{\delta\rightarrow0+}\frac{\nu\left(\widehat{\pi}^{-1}(\mathcal{P}(\widehat{\pi}(\ii),\theta)\cap \pi^{-1}(B(\pi(\ii),\delta)))\right)}{\nu(\pi^{-1}B(\pi(\ii),\delta))}\\
	&=\lim_{\delta\rightarrow0+}\frac{\nu\left([i_{0}]\cap \pi^{-1}(B(\pi(\ii),\delta)))\right)}{\nu(\pi^{-1}B(\pi(\ii),\delta))}=\nu_{\ii}^{\pi^{-1}}([i_{0}])
	\end{align*}
	for $\nu$-almost every $\ii$, we have finished the proof.
\end{proof}

\section{Totally dominated splitting and invertible system}\label{sec:Liftd}

Now we turn to the proof of Theorem~\ref{texactd}. In this section, we define another dynamical system, which is invertible and hyperbolic. Basically, this can be defined by relying on the TDS condition. In the first part of the section, we will list some properties of matrices satisfying the TDS condition, and in the second part, we give the basic definition of the mentioned dynamical system and its invariant partitions. We remark that the partitions are different from the ones defined in Section~\ref{sec:Lift}. Nevertheless, the essence of the proof of Theorem~\ref{texactd} is the same as that of Theorem \ref{texact2}. The main difference can be seen by comparing Propositions \ref{pLY2} and \ref{pLY2d}.

Let us first introduce the two-sided symbolic dynamics. Let  $\Sigma^{\pm}=\left\{1,\dots,N\right\}^{\mathbb{Z}}$ be the space of two-sided infinite words and let $\Sigma^-=\left\{1,\dots,N\right\}^{\mathbb{Z}\setminus\mathbb{N}}$ be the set of left-sided infinite words. Recall that the set of right-sided infinite words is $\Sigma$. Denote the left-shift operator on $\Sigma^{\pm}$ by $\sigma$ and the right-shift operator on $\Sigma^{\pm}$ and $\Sigma^-$ by $\sigma_-$. Thus, $\sigma$ and $\sigma_-$ are invertible on $\Sigma^{\pm}$ and $\sigma^{-1}=\sigma_-$.
	
For a two-sided infinite word $\ii\in\Sigma^{\pm}$ we denote by $\ii|_n^k$ the elements of $\ii$ between $n$ and $k$, i.e.\ $\ii|_n^k=(i_n,\dots,i_k)$. Let us also define the cylinder sets on $\Sigma^{\pm}$ by $$[\ii|_n^k]=\{\jj\in\Sigma^{\pm}:\jj|_n^k=\ii|^k_n\}.$$
For a word $\ii=(\cdots i_{-2}i_{-1}i_0i_1\cdots)\in\Sigma^{\pm}$, denote by $\ii_+=(i_0i_1\cdots) \in \Sigma$ the right-hand side and by $\ii_-=(\cdots i_{-2}i_{-1}) \in \Sigma^-$ the left-hand side of $\ii$. To avoid any confusion, we write $\ii_+$ for elements in $\Sigma$ and $\ii_-$ for elements in $\Sigma^-$. We define the projection from $\Sigma^{\pm}$ onto $\Sigma$ by $p_+\colon\Sigma^{\pm}\to\Sigma$, and similarly, the projection from $\Sigma^{\pm}$ onto $\Sigma^-$ by $p_-\colon\Sigma^{\pm}\to\Sigma^-$. Thus, $p_+(\ii)=\ii_+$ and $p_-(\ii)=\ii_-$.

\subsection{Totally dominated splitting}
In this section, we collect the results of Bochi and Guermelon \cite{BG} on dominated splitting. Let $\Alpha=\left\{A_1, \dots, A_N\right\}$ be a finite set of contractive non-singular $d\times d$ matrices. We define a mapping $A\colon\Sigma^{\pm}\to\mathcal{A}$ by setting $A(\ii)=A_{i_0}$. For $n\geq1$ we let $A^{(n)}(\ii)=A(\sigma^{n-1}\ii)\cdots A(\ii)$, for $n\leq-1$ we let $A^{(n)}(\ii)=A^{(-n)}(\sigma^{n}\ii)$, and for $n=0$ we let $A^{(0)}(\ii)=\mathrm{Id}$.
Recall that $\Alpha$ has \textit{dominated splitting} of index $i \in \{1,\ldots,d-1\}$ if there exist constants $C>0$ and $\tau<1$ such that
$$
\frac{\alpha_{i+1}(A^{(n)}(\ii))}{\alpha_{i}(A^{(n)}(\ii))}<C\tau^n
$$
for every $\ii\in\Sigma^\pm$ and $n\in\mathbb{Z}$. The following theorem is a refinement of Proposition \ref{thm:BGprop}.

\begin{theorem}[Bochi and Gourmelon \cite{BG}]\label{tdomsplitindex}
  Let $\Alpha=\left\{A_1, \dots, A_N\right\}$ be a finite set of contractive non-singular $d\times d$ matrices with dominated splitting of index $i \in \{ 1,\ldots,d-1 \}$. Then there exists a constant $C\geq1$ and for every $\ii\in\Sigma^{\pm}$ there are subspaces $F_{\ii}^i$ and $E_{\ii}^i$ such that
  \begin{itemize}
          \item[(1)] $\dim F_{\ii}^i=d-i$, $\dim E_{\ii}^i=i$, and $F_{\ii}^i\oplus E_{\ii}^i=\R^d$,
          \item[(2)] $A(\ii)F_{\ii}^i=F_{\sigma\ii}^i$ and $A(\ii)E_{\ii}^i=E_{\sigma\ii}^i$,
          \item[(3)] $F_{\ii}^i$ depends only on $\ii_+$ and $E_{\ii}^i$ depends only on $\ii_-$,
          \item[(4)] $\|A^{(n)}(\ii)|F_{\ii}^i\|\leq C\alpha_{i+1}(A^{(n)}(\ii))$ and $\mathfrak{m}(A^{(n)}(\ii)|E_{\ii}^i)\geq C^{-1}\alpha_{i}(A^{(n)}(\ii))$ for every $n\geq1$.
  \end{itemize}
  Moreover the angle between $F_{\ii}^i$ and $E_{\ii}^i$ is bounded below uniformly for every $\ii\in\Sigma^{\pm}$.
\end{theorem}

\begin{proof}
  The claim is a direct consequence of \cite[Theorem~A, Theorem~B, and Lemma~1]{BG}.
\end{proof}

Recall that $\Alpha$ satisfies the \emph{totally dominated splitting} (TDS) if for every $i\in \{ 1,\dots,d-1 \}$ either $\mathcal{A}$ has dominated splitting of index $i$ or there exists a constant $C\geq1$ such that
$$
C^{-1}\leq\frac{\alpha_{i+1}(A^{(n)}(\ii))}{\alpha_{i}(A^{(n)}(\ii))}
$$
for every $\jj \in \Sigma^\pm$ and $n \in \mathbb{Z}$. Recall also that $\mathcal{D}=\mathcal{D}(\mathcal{A}) = \{ i \in \{ 1,\ldots,d-1 \} : \mathcal{A}$ has dominated splitting of index  $i \}$.

\begin{lemma}\label{lTDSmanif}
  Let $\Alpha=\left\{A_1, \dots, A_N\right\}$ be a finite set of contractive non-singular $d\times d$ matrices satisfying the TDS. If the elements of $\mathcal{D}(\mathcal{A})$ are $1\leq i_1<\cdots<i_k\leq d-1$, then
  $$
  F^{i_k}_{\ii}\subset F^{i_{k-1}}_{\ii}\subset\cdots\subset F^{i_{1}}_{\ii}\quad\text{and}\quad E^{i_1}_{\ii}\subset E^{i_2}_{\ii}\subset\cdots\subset E^{i_k}_{\ii}.
  $$	
  Furthermore, if $i_0=0$ and $i_{k+1}=d$, then there exists a constant $C\geq1$ such that for every $j\in\{1,\ldots,k+1\}$ and $\ii\in\Sigma^{\pm}$ there are subspaces $e^{i_j}_{\ii}$ for which
  \begin{itemize}
          \item[(1)] $\dim e^{i_j}_{\ii}=i_{j}-i_{j-1}$,
          \item[(2)] $A(\ii)e^{i_j}_{\ii}=e^{i_j}_{\sigma\ii}$
          \item[(3)] $C^{-1}\alpha_{i_j}(A^{(n)}(\ii))\leq\mathfrak{m}(A^{(n)}(\ii)|e^{i_j}_{\ii})\leq\|A^{(n)}(\ii)|e^{i_j}_{\ii}\|\leq C\alpha_{i_j}(A^{(n)}(\ii))$ for every $n\geq1$.
  \end{itemize}
  Moreover, the angles between the subspaces $e^i_{\ii}$, $\ii\in\Sigma^{\pm}$, are uniformly bounded below.
\end{lemma}

\begin{proof}
  The subspaces $F_{\ii}^{i_j}$ can be defined as the limit of the eigenspace of $\left((A^{n}(\ii))^TA^{(n)}(\ii)\right)^{1/2}$ associated to the eigenvalues $\alpha_{i_j+1}(A^{(n)}(\ii)), \dots,\alpha_d(A^{(n)}(\ii))$; see \cite[Claim at p.\ 225]{BG}. This proves the first assertion of the lemma.

  Let us define the subspaces $ e^{i_j}_{\ii}$ as follows:
  $$
  e^{i_1}_{\ii}=E^{i_1}_{\ii},\quad e^{i_j}_{\ii}=E^{i_j}_{\ii}\cap F^{i_{j-1}}_{\ii}\text{ for all $j \in \{2\ldots, k\}$,} \quad e^{i_{k+1}}_{\ii}=F^{i_k}_{\ii}.
  $$
  The properties (1) and (2) follow now from Theorem~\ref{tdomsplitindex}. On the other hand, since
  \begin{align*}
    \mathfrak{m}(A^{(n)}(\ii)|e^{i_j}_{\ii})&\geq\mathfrak{m}(A^{(n)}(\ii)|E^{i_j}_{\ii})\geq C^{-1}\alpha_{i_j}(A^{(n)}(\ii)), \\
    \|A^{(n)}(\ii)|e^{i_j}_{\ii}\|&\leq\|A^{(n)}(\ii)|F_{\ii}^{i_{j-1}}\|\leq C\alpha_{i_{j-1}+1}(A^{(n)}(\ii))\leq C^d\alpha_{i_{j}}(A^{(n)}(\ii))
  \end{align*}
  also the property (3) holds.
  Finally, we remark that the cases $i_1$ and $i_{k+1}$ are straightforward.
\end{proof}

\subsection{Induced invertible system}
In this section, we introduce a dynamical system, induced by iterated function systems of affinities, similarly to \cite{Ba}. We assume that the IFS satisfies the SSC. The overlapping case is then treated as in Section \ref{sec:Lift} -- this will be done in Section \ref{sec:proofLYd}. Since the invariant strong stable manifolds under the TDS can be characterized in an explicit way and depend continuously, we work with an invertible dynamical system, which is not the case for general finite set of matrices.

Let $\Phi=\left\{f_i(\xv)=A_i\xv+\tv_i\right\}_{i=1}^N$ be an IFS on $\R^d$, where each $\underline{t}_i\in\R^d$ and $\mathcal{A}=\left\{A_1,\dots,A_N\right\}$ is a finite family of contractive non-singular $d\times d$ matrices that satisfies the TDS. We denote the composition of functions of $\Phi$ for a finite length word $\underline{i}=(i_1,\dots,i_n)$ by $f_{\underline{i}}=f_{i_1}\circ\cdots\circ f_{i_n}$ and the self-affine set associated to $\Phi$ by $\Lambda$.
Let us define a dynamical system $F$ acting on $\Lambda\times\Sigma$ by setting
\begin{equation*} 
F(\xv,\ii_+)=(f_{i_0}(\xv),\sigma\ii_+).
\end{equation*}
For simplicity, we often write $\y=(\xv,\ii_+)$. Observe that $F$ is invertible because $\Phi$ satisfies the SSC. We note that here the inverse $F^{-1}$ plays the role of the mapping $G$ in the planar case.

Define $\pi\colon\Sigma^-\to\Lambda$ by
\begin{equation*}
\pi(\dots,i_{-2},i_{-1})=\lim_{n\rightarrow\infty}f_{i_{-1}}\circ \cdots\circ f_{i_{-n}}(0)=\sum_{n=1}^{\infty}A_{i_{-1}}\cdots A_{i_{-n+1}}\tv_{i_{-n}}.
\end{equation*}
It is easy to see that $F$ is conjugate to $\sigma$ by the projection $\Pi\colon\Sigma^{\pm}\to\Lambda\times\Sigma$, $\Pi(\ii)=(\pi(\ii_-),\ii_+)$. That is,
$\Pi\circ\sigma=F\circ\Pi$. If $\nu$ is a $\sigma$-invariant and ergodic measure on $\Sigma^{\pm}$, then the measure $\mu=(\Pi)_*\nu=\nu\circ\Pi^{-1}$ is $F$-invariant and ergodic.

We define a sequence of dynamically invariant foliations on $\Lambda\times\Sigma$ with respect to the stable directions. For any $i\in\mathcal{D}$ let $\xi^i$ be the foliation with respect to $F^i_{\ii}$. Let us denote the hyperplane of $\R^d$ containing $\xv$ parallel to $F^i_{\ii_+}$ by $P_i(\xv,\ii_+)$. That is, for any $(\xv,\ii_+)\in\Lambda\times\Sigma$ we let
\begin{equation*}
\xi^i(\xv,\ii_+)=\left\{(\yv,\jj_+)\in\Lambda\times\Sigma:\ii_+=\jj_+\text{ and }\xv-\yv\in F^i_{\ii_+}\right\}=(\Lambda\times\Sigma)\cap(P_i(\xv,\ii_+)\times\left\{\ii_+\right\}),
\end{equation*}
where $F^i_{\ii}$ is defined in Theorem~\ref{tdomsplitindex}. We define the stable partition to be
\begin{equation*}
\xi^0(\xv,\ii_+)=\Lambda\times\left\{\ii_+\right\}.
\end{equation*}
For simplicity, we introduce the convention $F_{\ii_+}^0=\R^d$. It is easy to see that
$$
\xi^{i_k}>\xi^{i_{k-1}}>\cdots\xi^{i_1}>\xi^0.
$$
By \eqref{econdcond}, for any $i,j\in\mathcal{D}$ with $i<j$ we have
\begin{equation*}
\mu^{\xi^i}_{\y}=\int\mu^{\xi^j}_{\z}\dd\mu^{\xi^i}_{\y}(\z).
\end{equation*}
Moreover, by the invariance of the subspaces $F^i_{\ii}$ and the contractivity of the functions $f_i$
$$
F^{-1}\xi^i>\xi^i
$$
for every $i\in\mathcal{D}\cup\left\{0\right\}$.
We recall that $(F^{-1}\xi^i)(\y)=F(\xi^i(F^{-1}(\y)))$.

Let us define the conditional entropy $H^i$ of $F^{-1}\xi^i$ with respect to $\xi^i$ in the usual way by setting
\begin{equation*} 
H^i=H(F^{-1}\xi^i|\xi^i)=-\int\log\cm{\y}{\xi^i}((F^{-1}\xi^i)(\y))\dd\mu(\y).
\end{equation*}
for all $i\in\mathcal{D}\cup\{0\}$. If $i\notin\mathcal{D}\cup\left\{0\right\}$, then we set $H^i=H^j$, where $j=\max\left(\left(\mathcal{D}\cup\left\{0\right\}\right)\cap\left\{n\leq i\right\}\right)$.
Thus,
$$
H^{d-1}\leq H^{d-2}\leq\cdots\leq H^1\leq H^0
$$
and $H^{i}<H^{i-1}$ if and only if $i\in\mathcal{D}$.

For the partitions $\xi^i$, $i\in\mathcal{D}$, the conditional measures can be defined by weak-* convergence. Let us denote the $i$th transversal ball with radius $\delta>0$ centered at $\y=(\xv,\ii_+)$ by $B^T_{i}(\y,\delta)$, i.e.
\begin{equation}\label{dtransball}
B^T_i(\xv,\ii_+,\delta)=\left\{(\yv,\jj_+)\in\Lambda\times\Sigma^+:\ii_+=\jj_+,\text{ }\mathrm{dist}(P_i(\xv,\ii_+),P_i(\yv,\jj_+))<\delta\right\},
\end{equation}
where $\mathrm{dist}$ denotes the usual Euclidean distance. If $i,j$ be two consecutive elements of $\mathcal{D}\cup\left\{0\right\}$ such that $i<j$ then
\begin{equation}\label{econddefd}
\mu^{\xi^0}_{(\xv,\ii_+)}=\lim_{n\rightarrow\infty}\frac{\left.\mu\right|_{\Lambda\times[\ii_+|_{0}^n]}}{\mu(\Lambda\times[\ii_+|_{0}^n])}\quad\text{and}\quad\mu_{\y}^{\xi^j}=\lim_{\delta\to0+}\frac{\left.\mu^{\xi^i}_{\y}\right|_{B^T_j(\y,\delta)}}{\mu^{\xi^i}_{\y}(B^T_j(\y,\delta))}
\end{equation}
provided that the limit exists; see \cite{Si}.

Let us define the natural partition of the system  by $\mathcal{P}=\left\{f_i(\Lambda)\times\Sigma^{+}\right\}_{i=1}^N$. It is easy to see that for every $i\in\mathcal{D}$
\begin{equation}\label{erefin}
\mathcal{P}\vee\xi^i=F^{-1}\xi^i.
\end{equation}
For $n\geq1$ let us recall that
$$
(\mathcal{P}_0^{n-1})(\y)=\mathcal{P}(x)\cap F(\mathcal{P}(F^{-1}(\y)))\cap\cdots\cap F^{-(n-1)}(\mathcal{P}(F^{n-1}(\y)).
$$
Now we prove a similar invariance for conditional measures like in Lemma~\ref{lem:invariance}

\begin{lemma}\label{linvcond}
  For every $i\in\mathcal{D}\cup\left\{0\right\}$ and measurable set $Q\subseteq\Lambda\times\Sigma$ we have
  $$
  \mu^{\xi^i}_{\y}(Q)=\mu_{F(\y)}^{F^{-1}\xi^i}(FQ)
  $$
  for $\mu$-almost every $\y\in\Lambda\times\Sigma$.
\end{lemma}

\begin{proof}
  First, we show the claim for $i=0$. By \eqref{econddefd} and $F$-invariance of $\mu$,
  \begin{equation}\label{efor0}
  \begin{split}
  \mu^{\xi^0}_{\y}(Q)&=\lim_{n\rightarrow\infty}\frac{\mu(Q\cap(\Lambda\times[\ii_+|_{0}^n])}{\mu(\Lambda\times[\ii_+|_{0}^n])}=\lim_{n\rightarrow\infty}\frac{\mu(F\left(Q\cap(\Lambda\times[\ii_+|_{0}^n])\right)}{\mu(F\left(\Lambda\times[\ii_+|_{0}^n]\right))}\\
  &=\lim_{n\rightarrow\infty}\frac{\mu(F(Q)\cap(\Lambda\times[\sigma\ii_+|_{0}^n])\cap\mathcal{P}(F(\y)))}{\mu((\Lambda\times[\sigma\ii_+|_{0}^n])\cap\mathcal{P}(F(\y)))}=\frac{\mu^{\xi^0}_{F(\y)}(F(Q)\cap\mathcal{P}(F(\y)))}{\mu^{\xi^0}_{F(\y)}(\mathcal{P}(F(\y)))}\\&=\mu^{F^{-1}\xi^0}_{F(\y)}(F(Q)),
  \end{split}
  \end{equation}
  where in the last two equations we used \eqref{econdcond} and \eqref{erefin}. Let us then assume that $i \in \mathcal{D}$. By using \eqref{efor0}, we similarly get
  $$
  \mu^{\xi^i}_{\y}(Q)=\lim_{\delta\rightarrow0+}\frac{\mu^{\xi^0}_{\y}(Q\cap B^T_i(\y,\delta))}{\mu^{\xi^0}_{\y}(B^T_i(\y,\delta))}=\lim_{\delta\rightarrow0+}\frac{\mu^{F^{-1}\xi^0}_{F(\y)}(F\left(Q\cap B^T_i(\y,\delta))\right)}{\mu^{F^{-1}\xi^0}_{F(\y)}(F\left(B^T_i(\y,\delta)\right))}.
  $$
  It is easy to see that with the constant $c=\max_{i\in\mathcal{S}}\|A_i^{-1}\|>0$, we have
  $$
  B^T_i(F(\y),c^{-1}\delta)\cap\mathcal{P}(F(\y))\subseteq F\left(B^T_i(\y,\delta)\right)\subseteq B^T_i(F(\y),c\delta)\cap\mathcal{P}(F(\y))
  $$
  for every $\delta>0$. Hence
  \begin{align*}
  \lim_{\delta\rightarrow0+}\frac{\mu^{F^{-1}\xi^0}_{F(\y)}(F\left(Q\cap B^T_i(\y,\delta))\right)}{\mu^{F^{-1}\xi^0}_{F(\y)}(F\left(B^T_i(\y,\delta)\right))}&=\lim_{\delta\rightarrow0+}\frac{\mu^{\xi^0}_{F(\y)}(F(Q)\cap\mathcal{P}(F(\y))\cap F(B^T_i(\y,\delta)))}{\mu^{\xi^0}_{F(\y)}(\mathcal{P}(F(\y))\cap F(B^T_i(\y,\delta)))}\\&
  =\frac{\mu^{\xi^i}_{F(\y)}(F(Q)\cap\mathcal{P}(F(\y)))}{\mu^{\xi^i}_{F(\y)}(\mathcal{P}(F(\y)))}=\mu^{F^{-1}\xi^i}_{F(\y)}(F(Q)).
  \end{align*}
  The proof is finished.
\end{proof}

\section{Proof of Ledrappier-Young formula with TDS}\label{sec:proofLYd}

This section is devoted to show the following theorem.

\begin{theorem}\label{tLYSSC}
  Let $\Phi=\left\{f_i(\xv)=A_i\xv+\tv_i\right\}_{i=1}^N$ be an IFS on $\R^d$ such that $\Alpha:=\left\{A_1, A_2,\dots, A_N\right\}$ is a finite set of contracting non-singular $d\times d$ matrices satisfying the TDS. Let us also assume that $\Phi$ satisfies the SSC. Then for every $\sigma$-invariant and ergodic measure $\nu$ on $\Sigma^{\pm}$, the measure $\cmd{\y}{\xi^i}$ is exact dimensional for every $i\in\mathcal{D}\cup\left\{0\right\}$ and $\mu$-almost every $\y$, where $\mu=\Pi_*\nu$. Moreover,
  $$
  \dim\cmd{\y}{\xi^i}=\frac{H^{d-1}}{\ly{d}}+\sum_{j=i}^{d-2}\frac{H^j-H^{j+1}}{\ly{j+1}}.
  $$
\end{theorem}

By applying the lifting argument used in Section~\ref{sec:Lift} for higher dimensional systems, we will prove Theorem~\ref{texactd} as a consequence of Theorem~\ref{tLYSSC} at the end of this section.
The proof of Theorem~\ref{tLYSSC} is decomposed into two propositions.

\begin{prop}\label{pLY1}
Under the assumptions of Theorem~\ref{tLYSSC}, if $m=\max\mathcal{D}$, then the measure $\cmd{\y}{\xi^m}$ is exact dimensional for $\mu$-almost every $\y$ and
        $$
        \dim\cmd{\y}{\xi^m}=\frac{H^m}{\ly{m+1}}=\frac{H^{d-1}}{\ly{d}}.
        $$
\end{prop}

We note that by the definition of $H^{d-1}$ and the TDS, $H^m/\ly{m+1}=H^{d-1}/\ly{d}$.
The proof of Proposition~\ref{pLY1} is analogous to the proof of Proposition~\ref{prop:1}. By replacing $\rho^n$ with $\alpha_m(A_{i_{-1}}\cdots A_{i_{-n}})$, for which by Lemma~\ref{lTDSmanif}
$$
C^{-1}\|A^{(n)}(\ii'_+)|F^m_{\ii'_+}\|\leq\alpha_m(A_{i_{-1}}\cdots A_{i_{-n}})\leq C\mathfrak{m}(A^{(n)}(\ii'_+)|F^m_{\ii'_+}),
$$
where $F^{-n}(\xv,\ii)=(\xv',\ii'_+)$, one can show a similar statement to that of Lemma~\ref{ltoLY1}. By replacing $\cm{x}{\xi^d}$ with $\cmd{\y}{\xi^m}$ and $G$ with $F^{-1}$ in the proof of Proposition~\ref{prop:1}, one is then easily able to prove Proposition~\ref{pLY1}. We omit the detailed proof.

\begin{prop}\label{pLY2d}
  Under the assumptions of Theorem~\ref{tLYSSC}, if $i,j$ be two consecutive elements of $\mathcal{D}\cup\left\{0\right\}$ such that $i<j$, then the measure $(\cmd{\y}{\xi^i})^T_{(F^{j}_{\ii_+})^{\bot}}$ is exact dimensional for $\mu$-almost every $\y=(\xv,\ii_+)$ and
  $$
  \dim(\cmd{\y}{\xi^i})^T_{(F^{j}_{\ii_+})^{\bot}}=\frac{H^i-H^j}{\ly{j}}.
  $$
\end{prop}

We note that the measure $(\cmd{\y}{\xi^i})^T_{(F^{j}_{\ii_+})^{\bot}}$ is the orthogonal projection of the measure $\cmd{\y}{\xi^i}$ onto the orthogonal complement of $F^{j}_{\ii_+}$. Since $i<j$, we have $F^{j}_{\ii_+}\subset F^{i}_{\ii_+}$ and $\dim((F^{j}_{\ii_+})^{\bot}\cap F^{i}_{\ii_+})\geq1$.
Hence the proof of Proposition~\ref{pLY2d} is significantly different to the proof of Proposition~\ref{pLY2}, because the subspace, where $(\cmd{\y}{\xi^i})^T_{(F^{j}_{\ii_+})^{\bot}}$ is defined, can have strictly larger dimension than $1$. Therefore, we will give complete details.
Note that $\ly{j}=\ly{j-1}=\dots=\ly{i+1}$.

Let us modify the definition \eqref{dtransball} of the transversal ball $B^T_j(\y,\delta)$ and set
$$
B^t_j(\xv,\ii_+,\delta)=\{(\yv,\jj_+)\in\Lambda\times\Sigma:\ii_+=\jj_+\text{ and }\mathrm{dist}(e^{j}_{(\xv,\ii_+)}\cap P_j(\xv,\ii_+),e^{j}_{(\xv,\ii_+)}\cap P_j(\yv,\jj_+))<\delta\},
$$
where $e^{j}_{(\xv,\ii_+)}=F^{i}_{\ii_+}\cap E^j_{\xv}$ by definition; see Lemma~\ref{lTDSmanif}. Since the subspaces $E_{\xv}^j$ and $F_{\ii_+}^j$ are uniformly transversal, there exists a constant $c>0$ such that
$$
B^t_j(\y,c^{-1}\delta)\subseteq B^T_j(\y,\delta)\subseteq B^t_j(\y,c\delta)
$$
for every $\y\in\Lambda\times\Sigma$ and $\delta>0$. Let us define functions $g^{j,n}(\y)=\cmd{\y}{\xi^{j}}(\mathcal{P}_0^{n-1}(\y))$ and $g_{\delta}^{j,n}(\y)=\cmd{\y}{\xi^{i}}(B^t_{j}(\y,\delta)\cap\mathcal{P}_0^{n-1}(\y))/\cmd{\y}{\xi^{i}}(B^t_{j}(\y,\delta))$.
By \eqref{econddefd}, $g_{\delta}^{j,n}\rightarrow g^{j,n}$ as $\delta\rightarrow0+$ for $\mu$-almost everywhere and, since $g_{\delta}^{j,n}$ is uniformly bounded, $g_{\delta}\rightarrow g$ in $L^1(\mu)$ as $\delta\rightarrow0+$.

The following lemma guarantees that we may apply Maker's Ergodic Theorem.

\begin{lemma}
        The function $\sup_{\delta>0}\{-\log g_{\delta}^{j,n}\}$ is in $L^1(\mu)$.
\end{lemma}

\begin{proof}
  As with Lemma \ref{ltoLY2b}, the proof is a slight modification of the proof of \cite[Lemma~3.6]{Ba}.
\end{proof}

Let us observe that
\begin{equation}\label{ebound1}
\begin{split}
\cmd{F^{-n}(\y)}{\xi^{i}}\biggl(B^t_{j}\biggl(F^{-n}(\y),\frac{\delta}{\|A^{(n)}(\ii_+)|e^j(\y)\|}\biggr)\biggr)&\leq\frac{\cmd{\y}{\xi^{i}}(B^t_{j}(\y,\delta)\cap\mathcal{P}_0^{n-1}(\y))}{\cmd{\y}{\xi^{i}}(\mathcal{P}_0^{n-1}(\y))}\\
&\leq\cmd{F^{-n}(\y)}{\xi^{i}}\biggl(B^t_{j}\biggl(F^{-n}(\y),\frac{\delta}{\mathfrak{m}(A^{(n)}(\ii_+)|e^j(\y))}\biggr)\biggr).
\end{split}
\end{equation}
Indeed,	by \eqref{econdcond}, \eqref{erefin}, and Lemma~\ref{linvcond},
\begin{equation}\label{econt1d}
\frac{\cmd{\y}{\xi^{i}}(B^t_{j}(\y,\delta)\cap\mathcal{P}_0^{n-1}(\y))}{\cmd{\y}{\xi^{i}}(\mathcal{P}_0^{n-1}(\y))}=\cmd{F^{-n}(\y)}{\xi^{i}}(F^{-n}(B^t_{j}(\y,\delta))).
\end{equation}
On the other hand, by the definition of $B^t_j(\y,\delta)$,
\begin{equation}\label{econt2d}
B^t_{j}\biggl(F^{-n}(\y),\frac{\delta}{\|A^{(n)}(\ii_+)|e^j(\y)\|}\biggr)\subseteq F^{-n}(B^t_{j}(\y,\delta))\subseteq B^t_{j}\biggl(F^{-n}(\y),\frac{\delta}{\mathfrak{m}(A^{(n)}(\ii_+)|e^j(\y))}\biggr).
\end{equation}

\begin{proof}[Proof of Proposition~\ref{pLY2}]
  First, we will show the upper bound. Let $\y=(\xv,\ii_+)\in\Lambda\times\Sigma$ be such that $\xv=\pi(i_{-1}i_{-2}\cdots)$ and let $n,k\geq1$. Then, by \eqref{econt2d},
  \begin{align*}
  \cmd{\y}{\xi^{i}}(B^t_{j}(\y,\delta))&=
  \cmd{F^{-nk}(\y)}{\xi^{i}}(F^{-nk}(B^t_{j}(\y,c_{n,k}\delta)))\cdot\prod_{l=1}^k\dfrac{\cmd{F^{-n(l-1)}(\y)}{\xi^{i}}(F^{-n(l-1)}(B^t_{j}(\y,c_{n,l-1}\delta)))}{\cmd{F^{-nl}(\y)}{\xi^{i}}(F^{-nl}(B^t_{j}(\y,c_{n,l}\delta)))}\\
  &\geq\cmd{F^{-nk}(\y)}{\xi^{i}}\biggl(B^t_{j}\biggl(F^{-nk}(\y),\frac{c_{n,k}\delta}{\|A_{i_{-1}}\cdots A_{i_{-nk}}|e^j_{F^{-nk}(\y)}\|}\biggr)\biggr)\\
  &\quad\quad\cdot\prod_{l=1}^k\dfrac{\cmd{F^{-n(l-1)}(\y)}{\xi^{i}}\biggl(B^t_{j}\biggl(F^{-n(l-1)}(\y),\frac{c_{n,l-1}\delta}{\|A_{i_{-1}}\cdots A_{i_{-n(l-1)}}|e^j_{F^{-n(l-1)}(\y)}\|}\biggr)\biggr)}{\cmd{F^{-nl}(\y)}{\xi^{i}}\biggl(F^{-n}\biggl(B^t_{j}\biggl(F^{-n(l-1)}(\y),\frac{c_{n,l}\delta}{\mathfrak{m}(A_{i_{-1}}\cdots A_{i_{-n(l-1)}}|e^j_{F^{-n(l-1)}(\y)})}\biggr)\biggr)\biggr)},
  \end{align*}
  where $c_{n,l}$ will be defined later. By \eqref{econt1d}, \eqref{ebound1}, and Lemma~\ref{lTDSmanif},
  \begin{align*}
  &\cmd{F^{-nl}(\y)}{\xi^{i}}\biggl(F^{-n}\biggl(B^t_{j}\biggl(F^{-n(l-1)}(\y),\frac{c_{n,l}\delta}{\mathfrak{m}(A_{i_{-1}}\cdots A_{i_{-n(l-1)}}|e^j_{F^{-n(l-1)}(\y)})}\biggr)\biggr)\biggr)\\
  &\quad=\dfrac{\cmd{F^{-n(l-1)}(\y)}{\xi^{i}}\biggl(B^t_{j}\biggl(F^{-n(l-1)}(\y),\frac{c_{n,l}\delta}{\mathfrak{m}(A_{i_{-1}}\cdots A_{i_{-n(l-1)}}|e^j_{F^{-n(l-1)}(\y)})}\biggr)\cap\mathcal{P}_0^{n-1}(F^{-n(l-1)}(\y))\biggr)}{\cmd{\y}{\xi^{i}}(\mathcal{P}_0^{n-1}(F^{-n(l-1)}(\y)))}\\
  &\quad\leq\dfrac{\cmd{F^{-n(l-1)}(\y)}{\xi^{i}}\biggl(B^t_{j}\biggl(F^{-n(l-1)}(\y),\frac{c_{n,l}\delta}{C^{-1}\|A_{i_{-1}}\cdots A_{i_{-n(l-1)}}|e^j_{F^{-n(l-1)}(\y)}\|}\biggr)\cap\mathcal{P}_0^{n-1}(F^{-n(l-1)}(\y))\biggr)}{\cmd{\y}{\xi^{i}}(\mathcal{P}_0^{n-1}(F^{-n(l-1)}(\y)))}.
  \end{align*}
  By choosing $\delta=\mathfrak{m}(A_{i_{-1}}\cdots A_{i_{-nk}}|e^j_{F^{-nk}(\y)})$, we have
  \begin{multline*}
  \cmd{F^{-n(l-1)}(\y)}{\xi^{i}}\biggl(B^t_{j}\biggl(F^{-n(l-1)}(\y),\frac{c_{n,l-1}\delta}{\|A_{i_{-1}}\cdots A_{i_{-n(l-1)}}|e^j_{F^{-n(l-1)}(\y)}\|}\biggr)\biggr)\\
  \geq\cmd{F^{-n(l-1)}(\y)}{\xi^{i}}(B^t_{j}(F^{-n(l-1)}(\y),c_{n,l-1}C^{-1}\mathfrak{m}(A_{i_{-n(l-1)-1}}\cdots A_{i_{-nk}}|e^j_{F^{-nk}(\y)})))
  \end{multline*}
  and
  \begin{multline*}
  \cmd{F^{-n(l-1)}(\y)}{\xi^{i}}\biggl(B^t_{j}\biggl(F^{-n(l-1)}(\y),\frac{c_{n,l}\delta}{C^{-1}\|A_{i_{-1}}\cdots A_{i_{-n(l-1)}}|e^j_{F^{-n(l-1)}(\y)}\|}\biggr)\cap\mathcal{P}_0^{n-1}(F^{-n(l-1)}(\y))\biggr)\\
  \leq\cmd{F^{-n(l-1)}(\y)}{\xi^{i}}(B^t_{j}(F^{-n(l-1)}(\y),\frac{c_{n,l}}{C^{-1}}\mathfrak{m}(A_{i_{-n(l-1)-1}}\cdots A_{i_{-nk}}|e^j_{F^{-nk}(\y)}))\cap\mathcal{P}_0^{n-1}(F^{-n(l-1)}(\y))).
  \end{multline*}
  Set $c_{n,l}=c_{n,l-1}C^{-2}$, i.e.\ $c_{n,l}=C^{-2l}$ for $l\in\{0,\dots,k\}$. Then let us define
  $$
  h_{n,k,l}(\y)=\dfrac{\cmd{\y}{\xi^{i}}(B^t_{j}(\y,C^{-2l+1}\mathfrak{m}(A_{i_{-1}}\cdots A_{i_{-n(k-l)-1}}|e^j_{F^{-n(k-l)}(\y)}))\cap\mathcal{P}_0^{n-1}(\y))}{\cmd{\y}{\xi^{i}}(B^t_{j}(\y,C^{-2l+1}\mathfrak{m}(A_{i_{-1}}\cdots A_{i_{-n(k-l)-1}}|e^j_{F^{-n(k-l)}(\y)})))}.
  $$
  Then
  \begin{equation}\label{epartialupper}
  \begin{split}
  \frac{1}{k}\log \cmd{\y}{\xi^{i}}(B^t_{j}(\y,\mathfrak{m}(A_{i_{-1}}\cdots A_{i_{-nk}}|e^j_{F^{-nk}(\y)})))&\geq
  \frac{1}{k}\log\cmd{F^{-nk}(\y)}{\xi^{i}}(B^t_{j}(F^{-nk}(\y),C^{-2k}))\\
  &\quad\quad+\frac{1}{k}\sum_{l=0}^{k-1}\log\cmd{F^{-nl}(\y)}{\xi^{i}}(\mathcal{P}_0^{n-1}(F^{-nl}(\y)))\\&\quad\quad-\frac{1}{k}\sum_{l=0}^{k-1}\log h_{n,k,l}(F^{-nl}(\y)).
  \end{split}
  \end{equation}
  Since $F$ is conjugated to the full shift, $F^{-n}$ is ergodic. Thus, by applying Maker's Ergodic Theorem (Theorem~\ref{lmaker}) and Birkhoff's Ergodic Theorem,
  \begin{align}
  \lim_{k\to\infty}\frac{1}{k}\sum_{l=0}^{k-1}\log\cmd{F^{-nl}(\y)}{\xi^{i}}(\mathcal{P}_0^{n-1}(F^{-nl}(\y)))&=\int\log\cmd{\z}{\xi^{i}}(\mathcal{P}_0^{n-1}(\z))\dd\mu(\z)=-nH^i,\label{elim1}\\
  \lim_{k\to\infty}\frac{1}{k}\sum_{l=0}^{k-1}\log h_{n,k,l}(F^{-nl}(\y))&=\int\log\cmd{\z}{\xi^{j}}(\mathcal{P}_0^{n-1}(\z))\dd\mu(\z)=-nH^j\label{elim2}
  \end{align}
  for $\mu$-almost every $\y$.
  Let us now consider the first summable of the right-hand side of \eqref{epartialupper}. Let $m_k$ be the smallest integer such that $C^{-2k}\geq\|A_{i_{-nk-1}}\cdots A_{i_{-nk-m_k}}|e^j_{F^{-nk-m_k}(\y)}\|$. It is easy to see that there is an integer $\alpha>0$ such that $m_k\leq \alpha k$. Thus,
  $$
  \cmd{F^{-nk}(\y)}{\xi^{i}}(B^t_{j}(F^{-nk}(\y),C^{-2k}))\geq\cmd{F^{-nk}(\y)}{\xi^{i}}(\mathcal{P}_0^{\alpha k}(F^{-nk}(\y))).
  $$
  Then, by \eqref{erefin} and Lemma~\ref{linvcond},
  $$
  \cmd{F^{-nk}(\y)}{\xi^{i}}(\mathcal{P}_0^{\alpha k}(F^{-nk}(\y)))=\frac{\cmd{\y}{\xi^{i}}(\mathcal{P}_0^{(\alpha+n) k}(\y))}{\cmd{\y}{\xi^{i}}\left(\mathcal{P}_0^{nk}(\y)\right)}.
  $$
  Thus, by the Shannon-McMillan-Breiman Theorem
  $$
  \frac{1}{k}\log\cmd{F^{-nk}(\y)}{\xi^{i}}(\mathcal{P}_0^{\alpha k}(F^{-nk}(\y)))=-(\alpha+n)H^i+nH^i
  $$
  for $\mu$-almost every $\y$.  Hence, by \eqref{epartialupper}, \eqref{elim1}, and \eqref{elim2},
  $$
  \overline{\dim}(\cmd{\y}{\xi^{i}},\y)\leq\limsup_{k\to\infty}\frac{\log \cmd{\y}{\xi^{i}}(B^t_{j}(\y,\mathfrak{m}(A_{i_{-1}}\cdots A_{i_{-nk}}|e^j_{F^{-nk}(\y)})))}{\log \mathfrak{m}(A_{i_{-1}}\cdots A_{i_{-nk}}|e^j_{F^{-nk}(\y)})}\leq\frac{(\alpha+n)H^i-nH^j}{n\ly{j}}
  $$
  for $\mu$-almost every $\y$ and every $n\geq1$. Thus, the upper bound follows.

  The proof of lower bound for the local dimension is analogous. We set $c_{n,l}=C^{2l}$ and $\delta=\|A_{i_{-1}}\cdots A_{i_{-nk}}|e^j_{F^{-nk}(\y)}\|$. Similarly to \eqref{epartialupper}, we get
  \begin{multline*}
  \frac{1}{k}\log \cmd{\y}{\xi^{i}}(B^t_{j}(\y,\|A_{i_{-1}}\cdots A_{i_{-nk}}|e^j_{F^{-nk}(\y)}\|))\\
  \leq\frac{1}{k}\log\cmd{F^{-na_nk}(\y)}{\xi^{i}}(B^t_{j}(F^{-na_nk}(\y),C^{2a_nk}\|A_{i_{-na_nk-1}}\cdots A_{i_{-nk}}|e^j_{F^{-nk}(\y)}\|))\\
  +\frac{1}{k}\sum_{l=0}^{a_nk-1}\log\cmd{F^{-nl}(\y)}{\xi^{i}}(\mathcal{P}_0^{n-1}(F^{-nl}(\y)))-\frac{1}{k}\sum_{l=0}^{a_nk-1}\log \widehat{h}_{n,k,l}(F^{-nl}(\y)),
  \end{multline*}
  where $a_n=\frac{-n\log\alpha_{\max}}{2(\log C-n\log\alpha_{\max})}$ (and $\alpha_{\max}=\max_{i}\|A_i\|$). Define
  $$
  \widehat{h}_{n,k,l}(\y)=\dfrac{\cmd{\y}{\xi^{i}}(B^t_{j}(\y,C^{2l+1}\|A_{i_{-1}}\cdots A_{i_{-n(k-l)-1}}|e^j_{F^{-n(k-l)}(\y)}\|)\cap\mathcal{P}_0^{n-1}(\y))}{\cmd{\y}{\xi^{i}}(B^t_{j}(\y,C^{2l+1}\|A_{i_{-1}}\cdots A_{i_{-n(k-l)-1}}|e^j_{F^{-n(k-l)}(\y)}\|))}.
  $$
  Since
  $$
  C^{2l+1}\|A_{i_{-1}}\cdots A_{i_{-n(k-l)-1}}|e^j_{F^{-n(k-l)}(\y)}\|\leq C^{a_nk}\alpha_{\max}^{n(1-a_n)k}\leq\alpha_{\max}^{\frac{n}{2}k}\to0
  $$
  as $k\to\infty$,
  we may apply Maker's Ergodic Theorem (Theorem~\ref{lmaker}) and therefore
  $$
  \underline{\dim}(\cmd{\y}{\xi^{i}},\y)\geq\frac{a_nnH^i-a_nnH^j}{n\ly{j}}
  $$
  for every $n\geq1$ which proves the lower bound.
\end{proof}

\begin{proof}[Proof of Theorem~\ref{tLYSSC}]
  We prove the statement by induction. For $\max\mathcal{D}$ the statement follows by Proposition~\ref{pLY1}. Let $i\in\mathcal{D}\cup\left\{0\right\}$ and let us assume that the statement holds for every $k\in\mathcal{D}$ that $k>i$. Let $j=\min\left\{\mathcal{D}\cap\left\{k>i\right\}\right\}$.
  Then one can show the induction step by replacing the measure $\cm{x}{i+1}$ with $\cmd{\z}{\xi^{j}}$ and the measure $\cm{x}{i}$ with $\cmd{\z}{\xi^{i}}$ in the proof of Proposition~\ref{prop:LY3}.
  The statement follows.
\end{proof}

\begin{prop}\label{pLY3}
  Under the assumptions of Theorem~\ref{tLYSSC}, for every $i\in\mathcal{D}$ the measure $(\cmd{\y}{\xi^0})^T_{(F^{i}_{\ii_+})^{\bot}}$ is exact dimensional for $\mu$-almost every $\y=(\xv,\ii_+)$ and
  $$
  \dim(\cmd{\y}{\xi^0})^T_{(F^{i}_{\ii_+})^{\bot}}=\sum_{k=0}^{i-1}\frac{H^k-H^{k+1}}{\ly{k+1}}.
  $$
  In particular, if $j=\max(\{k<i\}\cap\mathcal{D})$, then
  \begin{equation*}
  \dim(\cmd{\y}{\xi^0})^T_{(F^{i}_{\ii_+})^{\bot}}=\dim(\cmd{\y}{\xi^0})^T_{(F^{j}_{\ii_+})^{\bot}}+\frac{H^j-H^i}{\ly{i}}.
  \end{equation*}
\end{prop}

\begin{proof}
  We prove the statement by induction. For $i=\min\mathcal{D}$ the statement follows from Proposition~\ref{pLY2d}. Let $i\in\mathcal{D}$ and let us assume that the proposition holds for every $j<i$, $j\in\mathcal{D}$. The induction step can be proven as in the proof of Proposition~\ref{pLY4} by replacing $\cm{x}{\xi^i}$ with $\cmd{\y}{\xi^i}$, $F_i(\theta)$ with $F_{\ii_+}^i$, and $G$ with $F^{-1}$.
\end{proof}

\begin{proof}[Proof of Theorem~\ref{texactd}]
The induction step can be proven as in the proof of Proposition~\ref{pLY4} by replacing $\cm{x}{\xi^i}$ with $\cmd{\y}{\xi^i}$ and $F_i(\theta)$ with $F_{\ii_+}^i$.
\end{proof}


\begin{thebibliography}{99}
	
		\bibitem{A} L. Arnold: {\em Random Dynamical Systems}, Springer-Verlag Berlin, Heidelberg, 1998.


\bibitem{B} K. Bara\'nski: Hausdorff dimension of the limit sets of some planar geometric constructions, {\em Advances in Mathematics} {\bf 210}, (2007), 215-245.

\bibitem{Ba} B. B\'ar\'any: On the Ledrappier-Young formula for self-affine measures, {\em Math. Proc. Camb. Phil. Soc.} {\bf 159} no. 3 (2015), 405-432.

\bibitem{BR} B. B\'ar\'any and M. Rams: Dimension maximizing measures for self-affine systems, preprint, 2015. available at arXiv:1507.02829.

\bibitem{BPS} L. Barreira, Y. Pesin and J. Schmeling: Dimension and product structure of hyperbolic measures, {\em Ann. of Math. (2)} {\bf 149} (1999), no. 3, 755-783.

\bibitem{Be} T. Bedford: {\em Crinkly curves, Markov partitions and box dimensions in self-similar sets}, PhD Thesis, The University of Warwick, 1984.


\bibitem{BG} J. Bochi and N. Gourmelon: Some characterizations of domination, {\em Math. Z.} {\bf 263} (2009), no. 1, 221-231.



\bibitem{ER} J.-P. Eckmann and D. Ruelle: Ergodic theory of chaos and strange attractors, {\em Rev. Modern Phys.} {\bf 57} (1985), no. 3, part 1, 617-656.

\bibitem{F} K. Falconer: The Hausdorff dimension of self-affine fractals, {\em Math. Proc. Camb. Phil. Soc.} {\bf 103} (1988), 339-350.

\bibitem{F2} K. Falconer: The dimension of self-affine fractals II, {\em Math. Proc. Camb. Phil. Soc.} {\bf 111} (1992), 169-179.

\bibitem{Fb1} K. Falconer: {\em Fractal Geometry: Mathematical Foundations and Applications}, John Wiley and Sons, 1990.

\bibitem{Fb2} K. Falconer: {\em Techniques in fractal geometry}, John Wiley \& Sons, Ltd., Chichester, 1997.

\bibitem{FK1} K. Falconer and T. Kempton: Planar self-affine sets with equal Hausdorff, box and affinity dimensions, preprint, 2015, available at arXiv:1503.01270.

\bibitem{FK2} K. Falconer and T. Kempton: The dimension of projections of self-affine sets and measures, preprint, 2015, available at arXiv:1511.03556.

\bibitem{FaMa} K. Falconer and P. Mattila: Strong Marstrand theorems and dimensions of sets formed by subsets of hyperplanes, preprint, 2015, available at arXiv:1503.01284.


\bibitem{FLR} A.-H. Fan, K.-S. Lau and H. Rao: Relationships between Different Dimensions
of a Measure, {\em Monatsh. Math.} {\bf 135}, (2002), 191-201.

\bibitem{FH} D.-J. Feng and H. Hu: Dimension Theory of Iterated Function Systems,  {\em Comm. Pure Appl. Math.} {\bf 62} (2009), no. 11, 1435-1500.

\bibitem{Fra} J. M. Fraser: On the packing dimension of box-like self-affine sets in the plane, {\em Nonlinearity} {\bf 25} (2012), no. 7, 2075-2092.

\bibitem{FS} J. M. Fraser and P. Shmerkin: On the dimensions of a family of overlapping self-affine carpets, {\em Ergodic Theory Dynam. Systems}, DOI:10.1017/etds.2015.21, 2015.


\bibitem{GL} D. Gatzouras and S. P. Lalley: Hausdorff and box dimension of certain self-affine fractals, {\em Indiana Univ. Math. J.} {\bf 41} (1992), 533-568.


\bibitem{HS} M. Hochman and B. Solomyak: On the dimension of Furstenberg measure for $ SL_2 (\R) $ random matrix products, preprint (2016), available at arXiv:1610.02641.

\bibitem{HL} I. Hueter and S. P. Lalley: Falconer's formula for the Hausdorff dimension of a self-affine set in $\R^2$, {\em Ergodic Theory Dynam. Systems} {\bf 15} (1995), no. 1, 77-97.

\bibitem{JPS} T. Jordan, M. Pollicott and K. Simon: Hausdorff dimension for randomly perturbed self affine attractors, {\em  Comm. Math. Phys}. {\bf 270} (2007), no. 2, 519-544.

\bibitem{K} A. K\"aenm\"aki: On natural invariant measures on generalised iterated function systems, {\em Ann. Acad. Sci. Fenn. Math.} {\bf 29} (2004), no. 2, 419-458.

\bibitem{KS} A. K\"aenm\"aki and P. Shmerkin: Overlapping self-affine sets of Kakeya type {\em Ergodic Theory Dynam. Systems} {\bf 29} (2009), no. 3, 941-965.

\bibitem{KV} A. K\"aenm\"aki and M. Vilppolainen: Dimension and measures on sub-self-affine sets. {\em Monatsh. Math.} {\bf 161} (2010), no. 3, 271-293.




\bibitem{KP} R. Kenyon and Y. Peres: Measures of full dimension on affine-invariant sets. {\em Ergodic Theory Dynam. Systems} {\bf 16} (1996), no. 2, 307-323.

\bibitem{LY1} F. Ledrappier and L.-S. Young: The metric entropy of diffeomorphisms. I. Characterization of measures satisfying Pesin's entropy formula. {\em Ann. of Math. (2)} {\bf 122} (1985), no. 3, 509-539.

\bibitem{LY2} F. Ledrappier and L.-S. Young: The metric entropy of diffeomorphisms. II. Relations between entropy, exponents and dimension. {\em Ann. of Math. (2)} {\bf 122} (1985), no. 3, 540-574.

\bibitem{M} P. T. Maker: Ergodic Theorem for a sequence of functions, {\em Duke Math. J.} {\bf 6} (1940), 27-30.

\bibitem{Ma} P. Mattila: {\em Geometry of sets and measures in Euclidean spaces. Fractals and rectifiability.} Cambridge University Press, Cambridge, 1995.


\bibitem{Mc} C. McMullen: The Hausdorff dimension of general Sierpiński carpets, {\em Nagoya Math. J.} {\bf 96} (1984), 1-9.

\bibitem{MS} I. D. Morris and P. Shmerkin: On equality of Hausdorff and affinity dimensions, via self-affine measures on positive subsystems, preprint (2016), available at arXiv:1602.08789.


\bibitem{PS} Y. Peres and W. Schlag: Smoothness   of   projections, Bernoulli convolutions, and the dimension of exceptions, {\em Duke Math. J.} {\bf 102} (2000), no. 2, 193-251.

\bibitem{PeSo} Y. Peres and B. Solomyak: Existence of {$L^q$} dimensions and entropy dimension for self-conformal measures, {\it Indiana Univ. Math. J.} {\bf 49} (2000), no. 4, 1603-1621.

\bibitem{Pi} A. Pinkus: {\em Totally positive matrices}, Cambridge University Press, 2010.

\bibitem{PU} F. Przytycki and M. Urbanski: On the Hausdorff dimension of some fractal sets, {\em Studia Math.} {\bf 93} (1989), no. 2, 155-186.

\bibitem{Rapa} A. Rapaport: On self-affine measures with equal Hausdorff and Lyapunov dimensions, preprint (2015), available at arXiv:1511.06893.

\bibitem{R} V. A. Rohlin: On the fundamental ideas of measure theory, {\em AMS Trans.} {\bf 10} (1962), 1-52.

\bibitem{Ro} E. Rossi: Local dimensions of measures on infinitely generated self-affine sets, {\em J. Math. Anal. Appl.} {\bf 413} (2014), no. 2, 1030-1039.

\bibitem{Sh} P. Shmerkin: Overlapping self-affine sets, {\em Indiana Univ. Math. J.} {\bf 55} (2006), no. 4, 1291-1331.

\bibitem{Si} D. Simmons: Conditional measures and conditional expectation; Rohlin's Disintegration Theorem, {\em Discrete Contin. Dyn. Syst.} {\bf 32} (2012), no. 7, 2565-2582.

\bibitem{S}  B. Solomyak: Measure and dimension for some fractal families, {\em Math. Proc. Camb. Phil. Soc.} {\bf 124}, (1998), no. 3,  531-546.

\ 

\end{thebibliography}
\end{document}